\documentclass[10pt,reqno]{amsart}

\usepackage{amsmath,bm}
\usepackage{amsthm}
\usepackage{amsfonts}
\usepackage{amssymb}
\usepackage{verbatim}
\usepackage{latexsym}
\usepackage[mathscr]{eucal}
\usepackage{color}
\usepackage{enumerate}
\usepackage[active]{srcltx}

\usepackage[colorlinks=true]{hyperref}
\hypersetup{urlcolor=blue, citecolor=red}

\numberwithin{equation}{section}

\title[Data Assimilation and Statistical Solutions]{Downscaling data assimilation algorithm with applications to statistical solutions of the Navier-Stokes equations}

\date{January 03, 2018}

\author[A. Biswas]{Animikh Biswas}
\address{\textnormal{(Animikh Biswas)} Department of Mathematics and Statistics\\
University of Maryland Baltimore County\\ Baltimore, MD 21250, USA.}
\email[A. Biswas] {abiswas@umbc.edu}

\author[C. Foias]{Ciprian Foias}
\address{\textnormal{(Ciprian Foias)} Department of Mathematics\\
Texas A\&M University\\ College Station, TX 77843, USA.}
\email[C. Foias] {foias@math.tamu.edu}

\author[C. F. Mondaini]{Cecilia F. Mondaini}
\address{\textnormal{(Cecilia F. Mondaini)} Department of Mathematics\\
Texas A\&M University\\ College Station, TX 77843, USA.}
\email[C. F. Mondaini] {cfmondaini@gmail.com}
 
\author[E. S. Titi]{Edriss S. Titi}
\address{\textnormal{(Edriss S. Titi)} Department of Mathematics\\
Texas A\&M University\\ College Station, TX 77843, USA. {\bf ALSO}, Department of Computer Science and Applied Mathematics, Weizmann Institute of Science, Rehovot 76100, Israel.}
\email[E. S. Titi] {titi@math.tamu.edu, edriss.titi@weizmann.ac.il}

\theoremstyle{plain}
\newtheorem{thm}{Theorem}[section]
\newtheorem{lem}{Lemma}[section]
\newtheorem{prop}{Proposition}[section]
\newtheorem{cor}{Corollary}[section]
\newtheorem{defs}{Definition}[section]
\theoremstyle{definition}
\newtheorem{rmk}{Remark}[section]

\newcommand{\cal}{\mathcal }
\newcommand{\comments}[1]{}

\newcommand{\locH}{{\mathrm{loc}_H}}

\newcommand{\R}{\mathbb R}
\newcommand{\nn}{\nonumber}
\newcommand{\D}{\displaystyle }

\newcommand{\dt}{{\D\frac{d}{dt}}}

\newcommand{\ra}{\rightarrow}
\newcommand{\lra}{\longrightarrow}

\newcommand{\N}{\mathbb N}
\newcommand{\be}{\begin{equation}}
\newcommand{\ee}{\end{equation}}
\newcommand{\bes}{\begin{equation*}}
\newcommand{\ees}{\end{equation*}}

\newcommand{\bs}{\boldsymbol}

\newcommand{\Om}{\Omega}

\newcommand{\ka}{\kappa}

\newcommand{\bu}{{\mathbf u}}
\newcommand{\bv}{{\mathbf v}}
\newcommand{\bw}{{\mathbf w}}
\newcommand{\bx}{{\mathbf x}}
\newcommand{\by}{{\mathbf y}}

\newcommand{\tw}{\widetilde{\bw}}
\newcommand{\tv}{\widetilde{\bv}}
\newcommand{\fT}{{\mathfrak T}}
\newcommand{\fTA}{{\mathfrak T}_{\cal A}}
\newcommand{\fTB}{{\mathfrak T}_{\textnormal{b}}}

\newcommand{\Cloc}{C_{\textnormal{loc}}}
\newcommand{\obv}{\overline{\bv}}
\newcommand{\obw}{\overline{\bw}}
\newcommand{\rd}{{\text{\rm d}}}
\newcommand{\f}{\mathbf{f}}
\newcommand{\g}{\mathbf{g}}
\newcommand{\mA}{\mathcal{A}}
\newcommand{\mB}{\mathcal{B}}

\newcommand{\mD}{\mathcal{D}}
\newcommand{\mE}{\mathcal{E}}
\newcommand{\mJ}{\mathcal{J}}
\newcommand{\mM}{\mathcal{M}}
\newcommand{\mN}{\mathcal{N}}
\newcommand{\mP}{\mathcal{P}}
\newcommand{\mV}{\mathcal{V}}

\newcommand{\mS}{\mathcal{S}}

\newcommand{\bea}{\begin{eqnarray}}
\newcommand{\eea}{\end{eqnarray}}
\newcommand{\beas}{\begin{eqnarray*}}
\newcommand{\eeas}{\end{eqnarray*}}
\newcommand{\mK}{\mathcal{K}}
\newcommand{\Cb}{C_{\textnormal{b}}}
\newcommand{\Lloc}{L_{\textnormal{loc}}}
\newcommand{\Lb}{L_{\textnormal{b}}}
\newcommand{\Lip}{\textnormal{Lip}}
\newcommand{\Ps}{\mathbb{P}_\sigma}
\newcommand{\per}{\textnormal{per}}
\newcommand{\bphi}{\boldsymbol{\varphi}}

\newcommand{\fTBH}{\fT_{\textnormal{b,loc}_{H}}}
\newcommand{\fTBV}{\fT_{\textnormal{b,loc}_{V}}}
\DeclareMathOperator{\Exp}{e}

\def\intav#1{\mathchoice
          {\mathop{\vrule width 6pt height 3 pt depth -2.5pt
                  \kern -9pt \intop}\nolimits_{\kern -6pt#1}}%
          {\mathop{\vrule width 5pt height 3 pt depth -2.6pt
                  \kern -6pt \intop}\nolimits_{#1}}%
          {\mathop{\vrule width 5pt height 3 pt depth -2.6pt
                  \kern -6pt \intop}\nolimits_{#1}}%
          {\mathop{\vrule width 5pt height 3 pt depth -2.6pt
                  \kern -6pt \intop}\nolimits_{#1}}}

\newcommand{\charfn}[1]{{\raisebox{1.2pt}{\mbox{$\chi
_{\kern-1pt\lower3pt\hbox{{$\scriptstyle{#1}$}}}$}}}}

\newcounter{enumaux}
\newcounter{commcount}
\stepcounter{commcount}

\begin{document}

\begin{abstract}
Based on a previously introduced downscaling data assimilation algorithm, which employs a nudging term to synchronize the coarse mesh spatial scales, we construct a {\it determining map} for recovering the full trajectories from their corresponding
coarse mesh spatial trajectories, and investigate its properties. This map is then used to develop a downscaling data assimilation scheme for statistical solutions of the two-dimensional Navier-Stokes equations, where the coarse mesh spatial statistics of the system is obtained from discrete spatial measurements. As a corollary, we deduce that statistical solutions for the Navier-Stokes equations are determined by their coarse mesh spatial distributions. Notably, we present our results in the context of the Navier-Stokes equations; however, the tools are general enough to be implemented for other dissipative evolution equations.
\end{abstract}



\maketitle

\section{Introduction}

For a given dissipative dynamical system, which is believed to accurately describe some aspect(s) of an underlying physical reality, often the problem of forecasting using the model becomes one of {\it initialization}. More precisely, one does not have the complete data available with which to properly initialize the system. However, in many cases, this is compensated by the fact that one has access to data from measurements of the system, collected continuously (or discretely) in time, albeit on a much coarser spatial grid than the desired resolution of the forecast.  This, for instance, is the case in atmospheric sciences where, since the launch of the first weather satellites in the 1960s, weather data has been collected nearly continuously in time, which furnishes us with the knowledge of the state of the system, e.g., the velocity vector field or temperature, on a {\it coarse spatial grid of points}. The objective of data assimilation and signal synchronization is to use this low spatial resolution observational measurements, obtained (nearly) continuously in time, to accurately find the corresponding reference solution, from which future predictions can be made. This has by now  wide ranging applications in atmospheric, oceanic, medical and biological sciences (see e.g. \cite{asch}, and references therein).

Due to its ubiquity in applications, several different types of methods have been 
developed for data assimilation; see, for instance, \cite{asch, Daley1991, HMbook2012, Kalnay2003, LSZbook2015, ReichCotterbook2015} and the references therein. Our focus here is on what is known as \emph{nudging} (or \emph{Newtonian relaxation}) method. There has been some earlier work implementing various versions of the nudging method in control theory and in the context of ordinary differential equations (ODEs); see, e.g., \cite{Nijmeijer2001,Thau1973}. Moreover, there has also been attempts of extending this approach to the context of partial differential equations (PDEs); see, e.g., \cite{Anthes74, hoke}. However, a proper and rigorous framework for the nudging approach was recently developed in \cite{AzouaniOlsonTiti2014}, where the authors consider a more general setting which is valid for a broad class of infinite-dimensional dissipative PDEs and observables. Although the results in \cite{AzouaniOlsonTiti2014} are obtained for the two-dimensional incompressible Navier-Stokes equations as the reference model and under the assumption of continuous in time and error-free measurements, later works applied this method to several other dissipative dynamical systems \cite{AlbanezNussenzveigLopesTiti2016,BiswasMartinez2017,FarhatJollyTiti2015, FarhatLunasinTiti2016a, FarhatLunasinTiti2016b, FLT,MarkowichTitiTrabelsi2016}, as well as to more general situations such as discrete in time and error-contaminated measurements (\cite{BessaihOlsonTiti2015,FoiasMondainiTiti2016}). Moreover, this method has also been shown to perform remarkably well in numerical simulations \cite{AltafTitiKnioZhaoMcCabeHoteit2015, FJJT, GeshoOlsonTiti2015, HOT}. A notable feature of the  approach in \cite{AzouaniOlsonTiti2014} is that it allows one to provide explicit conditions on the relaxation (nudging) parameter and the spatial resolution of the observations, in order to guarantee convergence of the algorithm to the reference solution of the model equation(s). Their idea is based on the fact that the long-time behavior of various (infinite-dimensional) dissipative dynamical systems, is determined by only a finite number of degrees of freedom, e.g., modes, nodes or volume elements \cite{CockburnJonesTiti1995, CockburnJonesTiti1997, FoiasProdi1967, FoiasTemam1984, FoiasTiti1991}.

The algorithm in \cite{AzouaniOlsonTiti2014} can be described as follows.  Suppose $\bu$ is a solution to a physical model over a domain $\Om$, whose time evolution is governed by the equation
    \begin{align}\label{reference:sys}
    \dt \bu =F(\bu),
    \end{align}
except that the initial data $\bu_0$ has \textit{not} been provided and is thus, unknown.  Then consider the following initial value problem:
    \begin{align}\label{feedback:sys}
        \dt \bw=F(\bw)-\beta J_h(\bw-\bu), \quad \bw(0)=\bw_0,
    \end{align}
where $\bw_0$ is \textit{any} given, initial condition, $h>0$, $\beta=\beta(h)>0$ is the
``relaxation/nudging parameter", and $J_h$ is an adequate finite-rank linear ``interpolant" operator, constructed from the observed coarse scale data (see \cite{AzouaniOlsonTiti2014, CockburnJonesTiti1995, CockburnJonesTiti1997}, and references therein, for details). Due to the fact that there exists finitely many determining coarse-mesh quantities for dissipative systems (e.g., modes, nodes, volume elements \cite{CockburnJonesTiti1995, CockburnJonesTiti1997, FoiasProdi1967, FoiasTemam1984, FoiasTiti1991}), it can be shown that the solution to \eqref{feedback:sys}, corresponding to the measurements interpolated by $J_h(\bu)$, converges exponentially to the solution $\bu$ of \eqref{reference:sys} (see \cite{AzouaniOlsonTiti2014}), provided $h <<1$ and $\beta(h)>>1$ are adequately chosen.

The downscaling data assimilation equation \eqref{feedback:sys} was crucial in constructing a map $W$ in \cite{FoiasJollyKravchenkoTiti2012, FoiasJollyKravchenkoTiti2014}, which we refer to as the {\it determining map}, from a ball in the space of all continuous, bounded functions on $\R$ with values in the (finite dimensional) space $J_h H^2$ to the phase space. This was then used in \cite{FoiasJollyKravchenkoTiti2014, FJLT} to construct an associated ODE termed the {\it determining form} for the Navier-Stokes equations \cite{FoiasJollyKravchenkoTiti2012, FoiasJollyKravchenkoTiti2014, FJLT}, and  subsequently several other evolutionary equations including surface quasi-geostrophic equation, the damped, driven nonlinear Schr\"{o}dinger equation and the Korteweg-de Vries equation \cite{JollyMartinezSadigovTiti2017, JollySadigovTiti2015, JollySadigovTiti2017}.

Our objective here is two fold. First, we study in more detail this determining map, including extending its domain of definition, which allows for rougher (in time) observations. In particular, we present a simpler
proof of the existence of the determining map $W$ with much more relaxed conditions than the previous construction in \cite{FoiasJollyKravchenkoTiti2012, FoiasJollyKravchenkoTiti2014}. This will allow us to apply it to the case where the observations are noisy or are contaminated with a random error. Additionally, we establish various properties including Frech\'{e}t differentiability of the map $W$. Our second objective is to construct a downscaling data assimilation algorithm to recover the statistics generated by the NSE from random initial data from the observed spatial coarse-mesh statistics. This is formalized by using the concept of statistical solutions, either on phase-space, introduced in \cite{Foias72, Foias73, FoiasProdi1976} (see also \cite{FMRT2001,FRT2010,FRT2013}), or trajectory space, introduced in \cite{VF78, VF88}. The map $W$ plays a fundamental role in developing such downscaling data assimilation algorithm for statistical solutions. As a corollary, we deduce that the notion of determining quantities, established before for individual trajectories, extends to the case of statistical solutions. We remark that other approaches to statistical data assimilation, for instance via Bayesian analysis and Kalman filtering, can be found in \cite{BLSZ2013, BLLMCSS2013, CLMMT2016, HMbook2012, LSZbook2015, ReichCotterbook2015} and the references therein. The statistical data assimilation technique introduced here can be viewed as providing a connection between the deterministic algorithm developed in \cite{AzouaniOlsonTiti2014} and the above mentioned approaches via Bayesian analysis and Kalman filtering. This issue will be further explored in a future work.

It is worth mentioning that we demonstrate our results in the context of the Navier-Stokes equations. However, the tools are general enough that they can be implemented equally to other dissipative evolution equations, such as the Rayleigh-B\'enard convection system and other geophysical models, such as the 3D primitive equations and planetary geostrophic models.

\section{Preliminaries}\label{secPreliminaries}

In this section, we briefly recall the background material and introduce some of the notation that is needed for the results presented later in sections \ref{secDetMap} and \ref{secStatSol}. In particular, we recall the basic setup and results regarding the 2D Navier-Stokes equations in the spatial periodic case; introduce the topologies and useful operators on spaces of continuous functions; and present the notation and properties of the interpolation operators used throughout this manuscript.


\subsection{Navier-Stokes equations and its functional setting}\label{subsecNSE}

The Navier-Stokes equations for a Newtonian, homogeneous and incompressible fluid in two dimensions are given by
\be \label{nse}
 \frac{\partial \bu}{\partial t} - \nu \Delta \bu + (\bu \cdot \nabla ) \bu + \nabla p = \g, \quad \nabla \cdot \bu = 0,
\ee
where $\bu = (u_1,u_2)$ and $p$ are the unknowns and denote the velocity vector field and the pressure, respectively; while $\nu > 0$ and $\g$ are given and denote the kinematic viscosity parameter and the body forces applied to the fluid per unit mass, respectively. We denote the spatial domain by $\Omega \subset \R^2$ and the time interval by $I \subset \R$. The variables $\bu$ and $p$ are functions of $(\bx,t) \in \Omega \times I$, while, for simplicity, we assume that $\g$ is a function of $\bx \in \Omega$ only and $\g \in (L^2(\Omega))^2$. However, similar results, taking into account the necessary details, are also valid for a function $\g = \g(\bx,t)$ with $\g \in L^\infty(I;(L^2(\Omega))^2)$.

For simplicity, and in order to fix ideas, we consider system \eqref{nse} with periodic boundary conditions, with $\Omega \subset \R^2$ as the basic domain of periodicity given by $\Omega = [0,L] \times [0,L]$, $L > 0$. Moreover, we assume that $\g$ has zero spatial average, i.e.,
\[
 \int_{\Omega}  \g (\bx) \rd \bx = 0.
\]

We now   proceed to recall the basic functional setting of the NSE, a systematic development of which can be found in \cite{bookcf1988, Temambook1997, Temambook2001}.
Let $\mV$ be the space of test functions, given by
\begin{multline}
 \mV = \left\{\bs{\varphi}: \R^2 \to \R^2 \, :\, \bs{\varphi}\, \text{ is a $L$-periodic trigonometric polynomial, } \right.\\
 \left. \nabla \cdot \bs{\varphi} = 0, \int_\Omega \bs{\varphi} (\bx) \rd \bx = 0 \right\}.
\end{multline}
We denote by $H$ and $V$ the closures of $\mV$ with respect to the norms in $(L^2(\Omega))^2$ and $(H^1(\Omega))^2$, respectively. Moreover, we denote by $H'$ and $V'$ the dual spaces of $H$ and $V$, respectively. As usual, we identify $H$ with $H'$, so that $V \subseteq H \subseteq V'$ with the injections being continuous and compact, with each space being densely embedded in the following one. The duality action between $V'$ and $V'$ is denoted by $\langle \cdot, \cdot \rangle_{V',V}$.

The inner product in $H$ is given by
\[
 (\bu_1,\bu_2) = \int_\Omega \bu_1 \cdot \bu_2 \, \rd \bx \quad \forall \bu_1, \bu_2 \in H,
\]
with the corresponding norm denoted by $\|\bu\|_{L^2} = (\bu,\bu)^{1/2}$. In $V$, we consider the following inner product:
\[
 (\!( \bu_1,\bu_2 )\!) = \int_\Omega \nabla \bu_1 : \nabla \bu_2 \, \rd \bx \quad \forall \bu_1, \bu_2 \in V,
\]
where it is understood that $\nabla \bu_1 : \nabla \bu_2$ denotes the component-wise product between the tensors $\nabla \bu_1$ and $\nabla \bu_2$. The corresponding norm in $V$ is given by $\|\nabla \bu\|_{L^2} = (\!( \bu, \bu )\!)^{1/2}$. The fact that $\|\nabla \cdot \|_{L^2}$ defines a norm on $V$ follows from the Poincar\'e inequality, given in \eqref{ineqPoincare}, below.

For every subspace $\Lambda \subset (L^1(\Omega))^2$, we denote
\[
 \dot{\Lambda}_\per = \left\{ \bs{\varphi} \in \Lambda \,:\, \bs{\varphi} \mbox{ is $L$-periodic and } \int_{\Omega} \bs{\varphi}(\bx) \rd \bx = 0 \right\}.
\]

Observe that $H$ is a closed subspace of $(\dot{L}^2(\Omega))^2$.
Let $P_\sigma$ denote the Helmholtz-Leray projector, which is defined as the orthogonal projection from $(\dot{L}^2_\per(\Omega))^2$ onto $H$. Applying $P_\sigma$ to \eqref{nse}, we obtain the following equivalent functional formulation
\be\label{eqNSE}
\frac{\rd \bu}{\rd t} + \nu A \bu + B(\bu,\bu) = \f \mbox{ in } V',
\ee
where $\f = P_\sigma \g$. The bilinear operator $B: V \times V \to V'$ is defined as the continuous extension of
\[
 B(\bu,\bv) = P_\sigma [(\bu \cdot \nabla) \bv] \quad \forall \bu, \bv \in \mV,
\]
and $A: D(A) \subset V \to V'$, the Stokes operator, is the continuous extension of
\[
 A \bu = - P_\sigma \Delta \bu \quad \forall \bu \in \mV.
\]
In fact, in the case of periodic boundary conditions, we have $A = - \Delta$.

We recall that $D(A) = V \cap (\dot{H}^2_\per(\Omega))^2$ and that $A$ is a positive and self-adjoint operator with compact inverse. Therefore, the space $H$ admits an orthonormal basis $\{\bs \phi_j\}_{j=1}^\infty$ of eigenfunctions of $A$ corresponding to a non-decreasing sequence of eigenvalues $\{\lambda_j\}_{j=1}^\infty$, with $\lambda_1:=\ka_0^2=(2\pi/L)^2$.

For each $N \in \N$, we consider the finite-dimensional space $H_N = \textnormal{span}\{\bs{\phi_1},$ $\ldots, \bs{\phi_N}\}$ and denote by $P_N$ the orthogonal projector of $H$ onto $H_N$.

We now recall some useful inequalities and identities. First, the Poincar\'e inequality, given by
\be\label{ineqPoincare}
 \kappa_0 \|\bu\|_{L^2} \leq \|\nabla \bu\|_{L^2} \quad \forall \bu \in V.
\ee
In two dimensions, the Ladyzhenskaya inequality is given by
\be\label{ineqLadyzhenskaya}
 \|\bu\|_{L^4} \leq c_L \|\bu\|_{L^2}^{1/2} \|\nabla \bu\|_{L^2}^{1/2} \quad \forall \bu \in (\dot{H}^1_\per(\Omega))^2,
\ee
and the Br\'ezis-Gallouet inequality,
\be\label{ineqBrezisGallouet}
 \|\bu\|_{L^\infty} \leq c_B \|\nabla \bu\|_{L^2} \left[1 + \log \left( \frac{\|A\bu\|_{L^2}}{\kappa_0 \|\nabla \bu\|_{L^2}} \right) \right]^{1/2} \quad \forall \bu \in D(A),
\ee
where $c_L$ and $c_B$ are nondimensional (scale invariant) constants and $\|\cdot\|_{L^4}$ and $\|\cdot\|_{L^\infty}$ denote the usual norms of the Lebesgue spaces $(L^4(\Omega))^2$ and $(L^\infty(\Omega))^2$.

The bilinear operator $B$ satisfies the following orthogonality property:
\[
 \langle B(\bu_1,\bu_2),\bu_3 \rangle_{V',V} = - \langle B(\bu_1,\bu_3),\bu_2 \rangle_{V',V} \quad \forall \bu_1, \bu_2, \bu_3 \in V.
\]
Moreover, in the case of periodic boundary conditions, we additionally have
\be \label{em}
 (B(\bu,\bu),A\bu) = 0 \quad \forall \bu \in D(A).
\ee

From Ladyzhenskaya inequality, \eqref{ineqLadyzhenskaya}, we obtain that
\be\label{estNonlTermC}
 \|B(\bu,\bv)\|_{L^2} \leq c_L^2 \|\bu\|_{L^2}^{1/2} \|\nabla\bu\|_{L^2}^{1/2} \|\nabla \bv\|_{L^2}^{1/2} \|A\bv\|_{L^2}^{1/2} \quad \forall \bu \in V, \bv \in D(A).
\ee
From Br\'ezis-Gallouet inequality, \eqref{ineqBrezisGallouet}, it follows that
\be\label{estNonlTermA}
 |(B(\bu_1,\bu_2),\bu_3)| \leq c_B \|\nabla \bu_1\|_{L^2} \|\nabla \bu_2\|_{L^2} \|\bu_3\|_{L^2} \left[1 + \log \left(\frac{\|A\bu_1\|_{L^2}}{\kappa_0 \|\nabla \bu_1\|_{L^2}} \right)\right]^{1/2}
\ee
for all $\bu_1 \in D(A), \bu_2 \in V, \bu_3 \in H$.
Moreover, the following logarithmic inequality was proved in \cite[Lemma 1.2, (c)]{Titi1987}:
\be\label{estNonlTermB}
 |(B(\bu_1, \bu_2),A\bu_3)| \leq c_T \|\nabla \bu_1\|_{L^2} \|\nabla \bu_2\|_{L^2} \|A\bu_3\|_{L^2}  \left[ 1 + \log \left(\frac{\|A\bu_2\|_{L^2}}{\kappa_0 \|\nabla\bu_2\|_{L^2}} \right) \right]^{1/2}
\ee
for all $\bu_1 \in V, \bu_2 \in D(A), \bu_3 \in D(A)$, with $c_T$ being a nondimensional constant.

It is well-known that, given $\bu_0 \in H$, there exists a unique solution $\bu$ of \eqref{eqNSE} on $(0,\infty)$ such that $\bu(0) = \bu_0$ and
\be\label{propssolNSE}
\bu \in C([0,\infty);H) \cap \Lloc^2([0,\infty);V)\ \mbox{and}\ \dt \bu \in L^2_{loc}([0,\infty); V').
\ee
Moreover, we also have $\bu \in C((0,\infty);D(A))$ (see, e.g., \cite[Theorem 12.1]{bookcf1988}). Therefore, equation \eqref{eqNSE} has an associated semigroup $\{S(t)\}_{t\geq 0}$, where, for each $t \geq 0$, $S(t): H \to H$ is the mapping given by
\be\label{defSt}
 S(t)\bu_0 = \bu(t),
\ee
with $\bu$ being the unique solution of \eqref{eqNSE} on $[0,\infty)$ satisfying $\bu(0) = \bu_0$ and \eqref{propssolNSE}.

Recall that a bounded set $\mB \subset H$ is called \emph{absorbing} with respect to $\{S(t)\}_{t \geq 0}$ if, for any bounded subset $B \subset H$, there exists a time $T = T(B)$ such that $S(t)B \subset \mB$ for all $t \geq T$. The existence of a bounded absorbing set for \eqref{eqNSE} is a well-known result. Then, the global attractor $\mA$ of \eqref{eqNSE} is defined as  the set satisfying any of the equivalent conditions given below.
\begin{enumerate}
\item Let $\mB \subset H$ be a bounded absorbing set with respect to $\{S(t)\}_{t \geq 0}$.
Then, the global attractor $\mA$ is given by
 \[
 \mA = \bigcap_{t\geq0} S(t) \mB.
\]

\item $\mA$ is the largest compact subset of $H$ which is invariant under the action of the semigroup $\{S(t)\}_{t\geq 0}$, i.e., $S(t) \mA = \mA$ for all $t \geq 0$.
\item $\mA$ is the minimal set that attracts all bounded sets.
\item $\mA$ is the set of all points in $H$ through which there exists a globally bounded trajectory $\bu(t), t \in \R$ with $\sup_{t \in \R} \|\bu(t)\|_{L^2} < \infty $.
\end{enumerate}

Also, recall the definition of the (dimensionless) Grashof number, given by
\be\label{defG}
 G = \frac{\|\f\|_{L^2}}{(\nu \kappa_0)^2}.
\ee

The following bounds hold in the global attractor $\mA$:
\be\label{boundmAH1}
 \|\nabla \bu\|_{L^2} \leq \nu \kappa_0 G \quad \forall \bu \in \mA,
\ee
and
\be\label{boundmADA}
 \|A\bu\|_{L^2} \leq c_2  \nu \kappa_0^2 (G + c_L^{-2})^3 \quad \forall \bu \in \mA,
\ee
where $c_2 = 2137 c_L^4$, with $c_L$ being the constant from \eqref{ineqLadyzhenskaya}. The proof of \eqref{boundmAH1} can be found in any of the references listed above (\cite{bookcf1988, Temambook1997, Temambook2001}) and the proof of \eqref{boundmADA} is given in \cite[Lemma 4.4]{FoiasJollyLanRupamYangZhang2015}.

\subsection{Spaces of continuous functions}\label{subsecSpacContFunc}

Given an interval $I \subset \R$  (we allow for the possibility that $I=\R$)
and a Banach space $Z$, we denote by $C(I;Z)$ the space of all continuous functions on $I$ with values in $Z$. Moreover, we denote by $\Cb(I;Z)$ the subset of $C(I;Z)$ consisting of bounded functions.

For every $t \in I$, we define the evaluation operator $\mE_t: C(I;Z) \to Z$, given by
\be
 \mE_t u = u(t) \quad \forall u \in C(I;Z).
\ee

For every subinterval $\tilde{I} \subset I$, we define the restriction operator $\mE_{\tilde{I}}: C(I;Z) \to C(\tilde{I};Z)$, as
\be
 \mE_{\tilde{I}} u(t) = u(t) \quad \forall t \in \tilde{I}.
\ee

Moreover, for every $I \subset \R$ and $\sigma \in \R$ such that $t + \sigma \in I$ for all $t \in I$, we define the translation operator $\tau_\sigma: C(I;Z) \to C(I;Z)$, given by
\be
(\tau_\sigma u)(t) = u(t + \sigma) \quad \forall t \in I, \quad \forall u \in C(I;Z).
\ee

\comments{
A useful topology on $C(I,Z)$ is the \emph{topology of uniform convergence on compact subsets} (see, e.g., \cite{Kelley1975}), which is defined as the topology generated by the subbase of neighborhoods of the form
\[
 \mN(K,O) =\{ u \in C(I,Z) \,:\, u(t) \in O \;\; \forall t \in K \},
\]
where $K \subset I$ is a compact subset and $O \subset Z$ is a neighborhood of the origin. Notably, a sequence $\{u_n\}_{n\in \mathbb{N}}$ in $C(I,Z)$ converges to $u \in C(I,Z)$ with respect to this topology if, and only if, for every compact subset $K \subset I$,
\[
 \sup_{t \in K} \|u_n(t) - u(t)\|_Z \to 0 \mbox{ as } n \to \infty,
\]
where $\|\cdot\|_Z$ denotes the norm in $Z$.

From now on, we use the notation $\Cloc(I,Z)$ to denote the space $C(I,Z)$ endowed with the topology of uniform convergence on compact subsets.
}

Also, for every $R > 0$, we denote by $B_Z(R)$ the closed ball in the Banach space $Z$ centered at $0$ with radius $R$.

\subsection{Interpolant Operators}

We recall an approach introduced in \cite{AzouaniOlsonTiti2014,AzouaniTiti2014} for dealing with several types of observables through a general class of interpolant operators. These operators are bounded, linear and of finite rank, and are required to satisfy an approximation of identity-type condition.

We consider two types of such operators. First, we say that $J: (\dot{H}^1(\Omega))^2 \to (\dot{L}^2(\Omega))^2$ is a \emph{Type I interpolant operator} if it satisfies
\be\label{type1}
 \|\bphi - J\bphi\|_{L^2} \leq c_1 h \|\nabla \bphi\|_{L^2} \quad \forall \bphi \in (\dot{H}^1(\Omega))^2.
\ee
Secondly, we say that $J : (\dot{H}^2(\Omega))^2 \to (\dot{L}^2(\Omega))^2$ is a \emph{Type II interpolant operator} if it satisfies
\be\label{type2}
 \|\bphi - J\bphi\|_{L^2} \leq c_{2,1} h\|\nabla \bphi\|_{L^2} + c_{2,2} h^2 \|\Delta \bphi\|_{L^2} \quad \forall \bphi \in (\dot{H}^2(\Omega))^2.
\ee
Here, $c_1$, $c_{2,1}$ and $c_{2,2}$ are absolute constants.

The idea is that, if $\bphi$ is the (unknown) state vector of a certain reference physical system, $J\bphi$ represents a spatial interpolation of the given spatial coarse-mesh measurements of $\bphi$. Thus, $J \bphi$ is known, while $\bphi$ is unknown.


A stronger property than \eqref{type2} is given by
\be\label{type2b}
 \|\bphi - J\bphi\|_{L^2} \leq c_{2,1}' h\|\nabla \bphi\|_{L^2} + c_{2,2}' h^{3/2}\|\nabla \bphi\|_{L^2}^{1/2} \|\Delta \bphi\|_{L^2}^{1/2} \quad \forall \bphi \in (\dot{H}^2(\Omega))^2,
\ee
where $c_{2,1}'$ and $c_{2,2}'$ are again absolute constants. Indeed, notice that \eqref{type2} follows from \eqref{type2b} by applying Young's inequality to the second term on the right-hand side. This stronger property for Type II interpolant operators is needed for obtaining some of the results concerning the ensemble data assimilation algorithm presented in subsection \ref{subsecEnsDA}.

Moreover, in sections \ref{secDetMap} and \ref{secStatSol}, in particular, we apply the operator $W$ (given in \eqref{defW}, below) to $J\bu(s), s \in \R$, with $\bu$ being a solution of \eqref{eqNSE} lying in the global attractor $\mA$. However, in order for this to make sense, we need to guarantee that $J\bu(s), s \in \R$ is in the domain of $W$, which is contained in a bounded subset of $\Cb(\R; (\dot{H}^1(\Omega))^2)$. With this aim, additional properties of $J$ must be assumed. Namely, in the case of a Type I interpolant, we assume in addition that $J((\dot{H}^1(\Omega))^2) \subset (\dot{H}^1(\Omega))^2$ and
that
\be\label{h1type1}
 \|\nabla(\bphi - J\bphi)\|_{L^2} \leq \tilde{c}_1 \|\nabla \bphi\|_{L^2} \quad \forall \bphi \in (\dot{H}^1(\Omega))^2),
\ee
for some absolute constant $\tilde{c}_1$. In the case of a Type II interpolant, we assume in addition that $J((\dot{H}^2(\Omega))^2) \subset (\dot{H}^1(\Omega))^2$ and
\be \label{h1type2}
  \|\nabla(\bphi - J\bphi)\|_{L^2} \leq \tilde{c}_{2,1} \|\nabla \bphi\|_{L^2} +\tilde{c}_{2,2}h \|\Delta \bphi\|_{L^2} \quad \forall \bphi \in  (\dot{H}^2(\Omega))^2,
\ee
for some absolute constants $\tilde{c}_{2,1}$ and $\tilde{c}_{2,2}$.

Examples of interpolant operators satisfying conditions \eqref{type1} and \eqref{h1type1} and, consequently, also \eqref{type2} and \eqref{h1type2}, include
\begin{enumerate}[(i)]
 \item the low-modes projector, i.e., $J = P_N$, for some $N = N(h) \in \N$;
 \item a spatial interpolation of observables given as averages over volume elements $Q_j$ of sidelength $h$, with $\Omega = \bigcup_{j=1}^{N} Q_j$ and $N = (L/h)^2$, defined as
\be\label{Jvolelts}
 (J\bphi)(\bx) = \sum_{j=1}^{N(h)} \frac{1}{|Q_j|} \int_{Q_j} \bphi(\by)\rd \by [ (\rho_\varepsilon \ast \chi_{Q_j})(\bx) - \langle \rho_\varepsilon \ast \chi_{Q_j}\rangle],
\ee
where $|Q_j| = h^2$ is the area of $Q_j$, $\chi_{Q_j}$ is the characteristic function of $Q_j$, $\rho_\varepsilon$ is a smooth mollifier with $\varepsilon = \varepsilon(h)$, and $\langle \cdot \rangle$ denotes the spatial average over $\Omega$, i.e., $\langle \psi\rangle = |\Omega|^{-1} \int_\Omega \psi \rd x$. Notably, the reason for subtracting the term $ \langle \rho_\varepsilon \ast \chi_{Q_j}\rangle$ is to obtain that $\langle J \bphi\rangle = 0$.
\end{enumerate}
Moreover, another example of interpolant operator satisfying \eqref{type2} and \eqref{h1type2} (but not necessarily \eqref{type1} and \eqref{h1type1}) is:
\begin{enumerate}[(i)]
 \setcounter{enumi}{2}
 \item a spatial interpolation of observables given as nodal values over volume elements $\{Q_j\}_{j=1}^{N}$ as above, defined as
\be\label{Jnodes}
 (J\bphi)(\bx) = \sum_{j=1}^{N(h)} \bphi(\bx_j) [ (\rho_\varepsilon \ast \chi_{Q_j})(\bx) - \langle \rho_\varepsilon \ast \chi_{Q_j}\rangle],
\ee
where $\bx_j \in Q_j$, $j = 1, \ldots, N$, are the nodal points and $\rho_\varepsilon$, $\chi_{Q_j}$ and $\langle \cdot\rangle$ are as above. In fact, one can show that such example of $J$ also satisfies the stronger property \eqref{type2b}. The proof of this fact is given in the Appendix.
\end{enumerate}

\section{The Determining Map}\label{secDetMap}

We now describe the lifting map introduced in \cite{FoiasJollyKravchenkoTiti2012, FoiasJollyKravchenkoTiti2014}, which played a crucial role in obtaining the determining form for the Navier-Stokes equations, as well as several other evolutionary equations \cite{JollyMartinezSadigovTiti2017, JollySadigovTiti2015, JollySadigovTiti2017}. In some sense, this map is proposed as a substitute, or alternative, for the notion of inertial manifold (see, e.g., \cite{bookcf1988, ConstantinFoiasNicolaenkoTemam1989, FoiasSellTemam1988, FoiasSellTiti1989, Temambook1997}), which is not known to exist for the NSE. Notably, the conditions we require for the construction of this map here are weaker than the ones considered in \cite{FoiasJollyKravchenkoTiti2014}. In particular, we do not require any condition on the time derivative of the input function $\bv = \bv(t)$ in \eqref{dataeqn}, below. This may be useful within the context of data assimilation, where $\bv(t)$ represents the observed spatial coarse-mesh data, which is usually noisy.

\subsection{The Determining Map $W$}\label{subsecDetMapW}

We start by defining the functional spaces where the domain and range of $W$ are contained, and by introducing the evolution equation which yields the definition of $W$.

First, we denote by $\Lb^2(\mathbb{R};D(A))$ the space of functions in $L^2(\mathbb{R};D(A))$ which are translation-bounded, i.e.
\[ \sup_{s\in\mathbb{R}} \int_s^{s+ \frac{1}{\nu \kappa_0^2}} \|A\bu(r)\|_{L^2}^2 \rd r < + \infty.
\]
The definition of translation-bounded functions is given, e.g., in \cite{ChepyzhovVishik2002}, with the slight difference that the upper limit of integration is written as ``$s+1$''. Here, we consider it as $s + (\nu \kappa_0^2)^{-1}$ in order to be dimensionally consistent.

Let
\[ Y = \Cb(\mathbb{R};V) \cap \Lb^2(\mathbb{R};D(A)).
\]
Notice that $Y$ is a Banach space with the norm
\[ \| \bu\|_Y = \left\{ \sup_{s\in \mathbb{R}} \frac{\|\nabla \bu(s)\|_{L^2}^2}{\nu^2 \kappa_0^2} + \sup_{s\in \mathbb{R}} \frac{1}{\nu \kappa_0^2} \int_s^{s + \frac{1}{\nu \kappa_0^2}} \|A\bu(r)\|_{L^2}^2 \rd r \right\}^{1/2}.
\]

Moreover, let $X$ be the Banach space
\[ X = \Cb(\mathbb{R}; (\dot{H}^1(\Omega))^2)
\]
equipped with the norm
\[ \| \bv\|_X = \sup_{s\in\mathbb{R}} \frac{\|\nabla \bv(s)\|_{L^2}}{\nu \kappa_0}.
\]

Given $\bv \in X$, and $\beta > 0$ a dimensionless parameter, consider the following evolution equation in functional form:
\be\label{dataeqn}
 \frac{\rd \bw}{\rd s} + \nu A \bw + B(\bw,\bw) = \f - \beta \nu \kappa_0^2 \Ps (J\bw - \bv), \quad s \in \mathbb{R},
\ee
where $\nu$ and $\f$ are the same as in \eqref{eqNSE}. Observe that \eqref{dataeqn} is not an initial value problem, but an evolution equation for all $s \in \R$.

\begin{rmk}
 Although we consider equation \eqref{dataeqn} with $\bv$ being a general element in the space $\Cb(\mathbb{R}; (\dot{H}^1(\Omega))^2)$, we emphasize that, for the specific applications we have in mind, we will restrict ourselves to the subspaces where $\bv$ takes values in the range of an interpolant operator $J$, i.e., $J((\dot{H}^1(\Omega))^2)$ in the case of a Type I interpolant, or $J((\dot{H}^2(\Omega))^2)$ in the case of a Type II interpolant.
\end{rmk}

The next proposition shows that, if $\bv \in B_X(\rho)$, for some $\rho > 0$, then, provided $\beta$ is large enough depending on the Grashoff number $G$, the parameter $h$ and $\rho$, with $h$ small enough depending on $\beta$ and $\rho$, the system \eqref{dataeqn} is well-posed.

In the following statement and in the remaining of this paper, we denote by $c$ a generic absolute constant, whose value may change from line to line.

\begin{prop}\label{propexistuniqw}
Let $J$ be either a Type I or Type II interpolant operator, i.e., satisfying either \eqref{type1} or \eqref{type2}. Let $\rho > 0$ and assume that $\beta > 0$ and $h>0$ satisfy
\be\label{condbeta}
\beta \geq c_1^*\left(\frac{G^2}{\beta} + \rho^2 \right) \log\left[c_1^* \left(\frac{G^2}{\beta} + \rho^2 \right) \right]
\ee
and
\be\label{condbetah}
\beta \kappa_0^2 h^2 \leq c_2^*.
\ee
Here, $c_1^* = c \max \{c_T^2, c_B^2\}$ and $c_2^* = c c_1^{-2}$ in case $J$ is a Type I interpolant satisfying \eqref{type1}, while $c_2^* = [c(c_{2,1}^2 + c_{2,2})]^{-1}$ in case $J$ is a Type II interpolant satisfying \eqref{type2}. Then, given $\bv \in B_X(\rho)$, there exists a unique global solution $\bw$ of \eqref{dataeqn} on $\mathbb{R}$ satisfying
\be\label{propssolw}
 \bw \in \Cb(\mathbb{R}; V) \cap \Lloc^2(\mathbb{R};D(A)), \quad \frac{\rd \bw}{\rd s} \in \Lloc^2(\mathbb{R};H).
\ee
Moreover, the following bounds hold:
\begin{enumerate}[(i)]
 \item\label{estw1} $\displaystyle \sup_{s\in \mathbb{R}} \|\nabla \bw(s)\|_{L^2}^2 \leq 2 (\nu \kappa_0)^2 \left( \frac{G^2}{\beta} + \|\bv\|_X^2 \right)$.
 \item\label{estw2}  $\displaystyle \sup_{s \in \mathbb{R}} \frac{1}{\nu \kappa_0^2} \int_s^{s+\frac{1}{\nu \kappa_0^2}} \|A\bw(r)\|_{L^2}^2 \rd r \leq 2 (1 + \beta) \left(\frac{G^2}{\beta} + \|\bv\|_X^2 \right)$.
 \setcounter{enumaux}{\value{enumi}}
\end{enumerate}
Also, consider $\bv_1, \bv_2 \in B_X(\rho)$ and let $\bw_1, \bw_2$ be the solutions of \eqref{dataeqn} on $\mathbb{R}$ corresponding to $\bv_1$ and $\bv_2$, respectively. Denote $\widetilde{\bw} = \bw_2 - \bw_1$ and $\widetilde{\bv} = \bv_2 - \bv_1$. Then,
\begin{enumerate}[(i)]
 \setcounter{enumi}{\value{enumaux}}
 \item\label{estdiffw1} $ \displaystyle \sup_{s\in \R} \|\nabla \widetilde{\bw}(s)\|_{L^2}^2 \leq 4 (\nu \kappa_0)^2 \|\widetilde{\bv}\|_X^2$.
 \item\label{estdiffw2} $ \displaystyle \sup_{s\in \R} \frac{1}{\nu \kappa_0^2} \int_s^{s + \frac{1}{\nu \kappa_0^2}} \|A\widetilde{\bw}(r)\|_{L^2}^2 \rd r \leq 4 (2 + \beta) \|\widetilde{\bv}\|_X^2$.
\end{enumerate}
In other words, the solution $\bw$ of \eqref{dataeqn} is unique and the map $W$ is locally Lipschitz and continuous in the $Y$-topology with respect to the input $\bv \in B_X(\rho)$.
\end{prop}
\begin{proof}
The existence of a global solution of \eqref{dataeqn} satisfying the properties in \eqref{propssolw} follows by deriving the estimates in \eqref{estw1} and \eqref{estw2} for
the unique solution $\bw_N$ of the Galerkin system \eqref{dataeqn} on the interval
$[-N (\nu \ka_0^2)^{-1}, \infty)$, subject to the initial value $\bw_N(-N (\nu \ka_0^2)^{-1})=0$, and then extracting a subsequence using the diagonal process and then passing to the limit, as in \cite{FoiasJollyKravchenkoTiti2014}.
  Since the details are given in \cite{FoiasJollyKravchenkoTiti2014}, we will omit them here, and restrict ourselves to proving estimates \eqref{estw1} and \eqref{estw2} formally. Moreover, the uniqueness of such solution will follow from the estimate in \eqref{estdiffw1}. Also, we show only the case of an interpolant operator $J$ satisfying \eqref{type2}, i.e., when $J$ is a Type II interpolant, since the case when $J$ is a Type I interpolant, satisfying \eqref{type1}, follows analogously.

Taking the inner product of \eqref{dataeqn} with $A \bw$ in $H$ and applying Cauchy-Schwarz, Young's inequality and property \eqref{type2} of $J$, we obtain that
\begin{multline}
   \frac12 \frac{\rd}{\rd s} \|\nabla \bw\|_{L^2}^2 + \nu \|A\bw\|_{L^2}^2
   = - \beta \nu \kappa_0^2 \|\nabla \bw\|_{L^2}^2 + \beta \nu \kappa_0^2 (\bv, A \bw) \\
   + (\f, A\bw) -  \beta \nu \kappa_0^2 (J \bw - \bw, A \bw)\\
   \leq - \beta \nu \kappa_0^2 \|\nabla \bw\|_{L^2}^2 + \beta \nu \kappa_0^2\|\nabla \bv\|_{L^2}^2 + \frac{\beta \nu \kappa_0^2}{4} \|\nabla \bw\|_{L^2}^2 + \frac{\|\f\|_{L^2}^2}{\nu} + \frac{\nu}{4} \|A\bw\|_{L^2}^2 \\
   + \frac{\beta \nu \kappa_0^2}{4} \|\nabla\bw\|_{L^2}^2 + (c_{2,1}^2 + c_{2,2}) \beta \nu \kappa_0^2 h^2 \|A\bw\|_{L^2}^2.
\end{multline}

Using hypothesis \eqref{condbetah} to estimate the last term in the right-hand side of the inequality above and rearranging the terms, we obtain
\be\label{eqwdotAw}
\frac{\rd}{\rd s} \|\nabla\bw\|_{L^2}^2 + \nu\|A\bw\|_{L^2}^2 + \beta \nu \kappa_0^2 \|\nabla\bw\|_{L^2}^2 \leq 2\beta \nu \kappa_0^2 \|\nabla \bv\|_{L^2}^2 + 2\frac{\|\f\|_{L^2}^2}{\nu}.
\ee

Integrating with respect to $s$ on the interval $[\sigma,t]$, we have
\be\label{estwVa}
 \|\nabla \bw(t)\|_{L^2}^2 \leq \| \nabla \bw(\sigma)\|_{L^2}^2 \Exp^{- \beta \nu \kappa_0^2 (t - \sigma)} + 2 \nu^2 \kappa_0^2 \left( \frac{G^2}{\beta} + \|\bv\|_X^2 \right).
\ee

Since $\bw \in \Cb(\mathbb{R}; V)$, taking the limit $\sigma \to -\infty$ in \eqref{estwVa} yields
\be\label{estwVb}
 \|\nabla \bw(t)\|_{L^2}^2 \leq 2 \nu^2 \kappa_0^2 \left( \frac{G^2}{\beta} + \|\bv\|_X^2 \right) \quad \forall t \in \mathbb{R},
\ee
which proves \eqref{estw1}. We stress again that this is a formal proof to establish the explicit bounds. In particular, the rigorous proof, using the Galerkin procedure, does not use the assumption that $\bw \in C_b(\R; V)$; but it establishes these estimates for the solution of the galerkin system $\bw_N(s)$, for $s \in [-N (\nu \ka-0^2)^{-1},\infty)$. The limit solution $W$ will enjoy these estimates for all $s \in \R$.

Now, ignoring the third term on the left-hand side of \eqref{eqwdotAw}, integrating both sides of the resulting inequality with respect to $s$ on the interval $[t, t+ (\nu \kappa_0^2)^{-1}]$ and using \eqref{estwVb}, we obtain
\be
 \frac{1}{\nu \kappa_0^2} \int_t^{t + \frac{1}{\nu \kappa_0^2}} \|A\bw(s)\|_{L^2}^2 \rd s \leq 2 (1 + \beta) \left( \frac{G^2}{\beta} + \|\bv\|_X^2 \right) \quad \forall t \in \mathbb{R},
\ee
proving \eqref{estw2}.

In order to prove \eqref{estdiffw1} and \eqref{estdiffw2}, let $\bw_1$ and $\bw_2$ be two solutions of \eqref{dataeqn} corresponding to functions $\bv_1$ and $\bv_2$ in $X$, respectively. Let $\widetilde{\bw} = \bw_2 - \bw_1$ and $\widetilde{\bv} = \bv_2 - \bv_1$. Then, $\widetilde{\bw}$ satisfies
\be\label{eqwtilde}
 \frac{\rd \tw}{\rd s} + \nu A\tw + B(\tw,\bw_1) + B(\bw_1,\tw) + B(\tw,\tw) = -\beta \nu \kappa_0^2 \Ps (J \tw - \tv).
\ee

Taking the inner product of \eqref{eqwtilde} with $A\tw$ in $H$ and applying Cauchy-Schwarz, Young's inequality and property \eqref{type2} of $J$, we have
\begin{multline}\label{ineqwtilde}
\frac{1}{2} \frac{\rd}{\rd s} \|\nabla \tw\|_{L^2}^2 + \nu \|A\tw\|_{L^2}^2 \leq -\beta \nu \kappa_0^2 \|\nabla \tw\|_{L^2}^2 \\
+ 2 \max\left\{ |(B(\tw,\bw_1),A\tw)|, |(B(\bw_1,\tw),A\tw)|\right\} + \beta \nu \kappa_0^2 (\tv, A\tw) \\
- \beta \nu \kappa_0^2 (J \tw - \tw , A \tw)\\
\leq -\beta \nu \kappa_0^2 \|\nabla \tw\|_{L^2}^2 + 2 \max\left\{ |(B(\tw,\bw_1),A\tw)|, |(B(\bw_1,\tw),A\tw)|\right\} \\
+ \frac{\beta \nu \kappa_0^2}{4} \|\nabla \tw\|_{L^2}^2 
+ \beta \nu \kappa_0^2 \|\nabla \tv\|_{L^2}^2
+ \frac{\beta \nu \kappa_0^2}{4} \|\nabla \tw\|_{L^2}^2 + \beta \nu \kappa_0^2 (c_{2,1}^2 + c_{2,2}) h^2 \|A\tw\|_{L^2}^2
\end{multline}

Notice that, due to \eqref{estNonlTermA} and \eqref{estNonlTermB}, we have
\begin{multline}\label{nonlinLip}
\max\left\{ |(B(\tw,\bw_1),A\tw)|, |(B(\bw_1,\tw),A\tw)|\right\} \\
 \leq c_{BT} \|\nabla\tw\|_{L^2}\|\nabla\bw_1\|_{L^2}\|A\tw\|_{L^2}\left[\log \left(\frac{e}{\kappa_0}\frac{\|A\tw\|_{L^2}}{\|\nabla \tw\|_{L^2}}\right)\right]^{1/2} \\
 \leq \frac{\nu}{6}\|A\tw\|_{L^2}^2 + \frac{3}{2} \frac{c_{BT}^2}{\nu} \|\nabla \bw_1\|_{L^2}^2 \|\nabla\tw\|_{L^2}^2  \left[1+ \log\left(\frac{\|A\tw\|_{L^2}^2}{\kappa_0^2\|\nabla\tw\|_{L^2}^2}\right)\right],
\end{multline}
where $c_{BT}^2 = \max\{c_T^2,c_B^2\}$.

Thus, using \eqref{nonlinLip} and \eqref{condbetah} in \eqref{ineqwtilde} and rearranging the terms, we obtain that
 \begin{multline}\label{wtildeineq}
 \frac{\rd}{\rd s} \|\nabla \tw\|_{L^2}^2 + \frac{\nu}{2}\|A\tw\|_{L^2}^2\\
 +
\left\{ \beta\nu \kappa_0^2 + \frac{\nu\kappa_0^2}{2}  \frac{\|A\tw\|_{L^2}^2}{\kappa_0^2\|\nabla\tw\|_{L^2}^2}-
  \frac{6 c_{BT}^2}{\nu} \|\nabla\bw_1\|_{L^2}^2 \left[1+ \log\left(\frac{\|A\tw\|_{L^2}^2}{\kappa_0^2\|\nabla\tw\|_{L^2}^2}\right)\right]\right\}\|\nabla \tw\|_{L^2}^2\\
 \leq 2 \beta \nu \kappa_0^2 \|\nabla\tilde{\bv}\|_{L^2}^2.
\end{multline}

We will now need the following elementary inequality (see \cite{AzouaniOlsonTiti2014}), namely,
\be\label{elemineq}
 \min_{r \ge 1}\{r - \zeta (1+ \log r)\}\ge - \zeta \log \zeta , \quad \zeta >0.
\ee

Since $\bw_1$ satisfies \eqref{estwVb} and $\bv_1 \in B_X(\rho)$, we have
\be\label{boundw1H1}
\frac{\|\nabla\bw_1(t)\|_{L^2}^2}{(\nu \kappa_0)^2} \leq K(G,\rho,\beta) := 2 \left( \frac{G^2}{\beta} + \rho^2 \right).
\ee

Thus, from \eqref{elemineq} and \eqref{boundw1H1}, we obtain
\begin{multline}\label{estfromelemineq}
 \frac{\nu\kappa_0^2}{2}  \frac{\|A\tw\|_{L^2}^2}{\kappa_0^2\|\nabla \tw\|_{L^2}^2} - \frac{6 c_{BT}^2}{\nu} \|\nabla \bw_1\|_{L^2}^2 \left(1+ \log\left(\frac{\|A\tw\|_{L^2}^2}{\kappa_0^2\|\nabla\tw\|_{L^2}^2}\right)\right) \\
 \geq \frac{\nu\kappa_0^2}{2} \left(  \frac{\|A\tw\|_{L^2}^2}{\kappa_0^2\|\nabla\tw\|_{L^2}^2}  - 12 c_{BT}^2 K \left(1+ \log\left(\frac{\|A\tw\|_{L^2}^2}{\kappa_0^2\|\nabla\tw\|_{L^2}^2}\right)\right) \right) \\
 \geq - 6 c_{BT}^2 \nu \kappa_0^2 K \log(12 c_{BT}^2 K).
\end{multline}

Using \eqref{estfromelemineq} in \eqref{wtildeineq} and hypothesis \eqref{condbeta}, we then have
\be\label{wtildeineq2}
\frac{\rd}{\rd s} \|\nabla\tw\|_{L^2}^2 + \frac{\nu}{2}\|A\tw\|_{L^2}^2 + \frac{\beta\nu \kappa_0^2}{2} \|\nabla\tw\|_{L^2}^2 \leq 2 \beta \nu \kappa_0^2 \|\nabla\tv\|_{L^2}^2.
\ee

Integrating both sides of \eqref{wtildeineq2} with respect to $s$ on the interval $[\sigma,t]$, it follows that
\be\label{estdiffw3a}
\|\nabla \tw(t)\|_{L^2}^2 \leq \|\nabla \tw(\sigma)\|_{L^2}^2 \Exp^{-\frac{\beta \nu \kappa_0^2}{2}(t-\sigma)} + 4 \sup_{s\in \mathbb{R}} \|\nabla \tv(s)\|_{L^2}^2.
\ee
Since $\bw_1,\bw_2 \in \Cb(\mathbb{R};V)$, taking the limit $\sigma \to -\infty$, we obtain
\be\label{estdiffw3}
 \|\nabla\tw(t)\|_{L^2}^2 \leq 4 (\nu \kappa_0)^2 \|\tv\|_X^2 \quad \forall t \in \mathbb{R},
\ee
which proves \eqref{estdiffw1}.

Now, integrating \eqref{wtildeineq2} with respect to $s$ on the interval $[t, t+ (\nu \kappa_0^2)^{-1}]$, we obtain
\[
\frac{1}{\nu \kappa_0^2} \int_t^{t + \frac{1}{\nu \kappa_0^2}} \|A\widetilde{\bw}(s)\|_{L^2}^2 \rd s \leq 4 (2 + \beta) \|\widetilde{\bv}\|_X^2 \quad \forall t \in \mathbb{R},
\]
proving \eqref{estdiffw2}.

\end{proof}

\begin{rmk}
 Notice that condition \eqref{condbeta} on $\beta$ is not used for proving items \eqref{estw1} and \eqref{estw2} of Proposition \ref{propexistuniqw}, but only for proving items \eqref{estdiffw1} and \eqref{estdiffw2}. In particular, this means that the existence of solutions to equation \eqref{dataeqn} can be proved by only assuming condition \eqref{condbetah}, but for proving uniqueness of such solutions, and the local Lipschitz continuity of the map $W$ with respect to the input $\bv$,
 one needs to assume, in addition, \eqref{condbeta}.
\end{rmk}


Now, the result of Proposition \ref{propexistuniqw} allows us to give the following definition.

\begin{defs}\label{defW0}
Consider $\rho > 0$, $\beta > 0$ and $h>0$ satisfying conditions \eqref{condbeta} and \eqref{condbetah}. Then, the \textbf{determining map} $W: B_X(\rho) \to Y$, is given by
\be\label{defW}
W(\bv) = \bw,
\ee
where $\bw$ is the unique solution of \eqref{dataeqn} corresponding to $\bv \in B_X(\rho)$.
\end{defs}

\begin{rmk}
 Observe that if $\bv_1,\bv_2 \in B_X(\rho)$ are such that $P_\sigma \bv_1 = P_\sigma \bv_2$, then, from \eqref{dataeqn} and the definition of $W$, it follows that $W(\bv_1) = W(\bv_2)$. Thus, for a given $\bv \in B_X(\rho)$, $W$ is in fact determined by $P_\sigma \bv$.
\end{rmk}


\subsection{Basic properties of the determining map $W$}

The next theorem summarizes some of the properties of the determining map $W$ given in Definition \ref{defW0}. These properties are essential for obtaining the results in Section \ref{secStatSol}.

Let $J$ be an interpolant operator of Type I or Type II. For every interval $I \subset \R$, we denote by $\mJ: C(I,(\dot{H}^1(\Omega))^2) \to C(I,(\dot{L}^2(\Omega))^2)$, in case $J$ is of Type I,
or $\mJ: C(I,(\dot{H}^2(\Omega))^2) \to C(I,(\dot{L}^2(\Omega))^2)$, in case $J$ is of Type II, the linear operator defined by
\be\label{defmJ}
 (\mJ \bu)(t) = J (\bu(t)) \quad \forall t \in I.
\ee

\begin{thm}\label{thmpropsW}
 Assume the hypotheses of Proposition \ref{propexistuniqw}. Then, the mapping $W: B_X(\rho) \to Y$ defined in \eqref{defW} satisfies the following properties:
 \begin{enumerate}[(i)]
  \item\label{thmpropsWi} For every $\bv \in B_X(\rho)$, $W(\bv) \in B_Y(\sqrt{M})$, with
    \[
      M := M(G,\rho,\beta) = 2(2+\beta)\left( \frac{G^2}{\beta} + \rho^2 \right).
    \]
  \item\label{thmpropsWii} $W$ is a Lipschitz mapping from $B_X(\rho)$ to $Y$ with Lipschitz constant $2(3+\beta)^{1/2}$.
  \item\label{thmpropsWiii} Let $\bv \in B_X(\rho)$ and $\bu$ be a solution of \eqref{eqNSE} on the global attractor $\mA$ (i.e., $\bu(t) \in \mA$ for all $t \in \mathbb{R}$) such that $\|\nabla J\bu(s) - \nabla\bv(s)\|_{L^2} \to 0$, as $s \to \infty$. Then $\|\nabla[W(\bv) - \bu](s)\|_{L^2} \to 0$, as $s \to \infty$.
  \item\label{thmpropsWiv} Let $\rho>0$ and $J$ be either a Type I or a Type II interpolant. In case $J$ is a Type I interpolant, assume that it additionally satisfies
  \eqref{h1type1}, and $\rho$ satisfies
    \be\label{condrhoa}
      \rho \geq (1 + \tilde{c}_1) G,
    \ee
    where $\tilde{c}_1$ is the constant from \eqref{h1type1} .
    In case $J$ is a Type II interpolant,  assume that it also satisfies \eqref{h1type2} and
    \be\label{condrho}
      \rho \geq c_3^* \left[ G + \frac{(G + c_L^{-2})^3}{\beta^{1/2}} \right],
    \ee
    where $c_3^* = \max\{1 + \tilde{c}_{2,1}, \tilde{c}_{2,2} c_2 \sqrt{c_2^*}\}$, with $\tilde{c}_{2,1}$ and $\tilde{c}_{2,2}$ being the constants from \eqref{h1type2}, $c_2$ the constant from \eqref{boundmADA} and $c_2^*$ the constant from \eqref{condbetah}. Under these hypotheses,
    if $\bu(s)$ for $ s \in \R$ is a solution of \eqref{eqNSE}, which is a trajectory in the  global attractor $\mA$, then
    \[
     W \circ \mJ (\bu) = \bu.
    \]
    Moreover, if $\mJ \circ W (\bv) = \bv$ for some $\bv \in B_X(\rho)$, then $W(\bv)$ is a solution of \eqref{eqNSE}, and it is a trajectory in the global attractor $\mA$.
  \item\label{thmpropsWv} Let $\bv_1,\bv_2 \in B_X(\rho)$. Then, $W(\bv_1) = W(\bv_2)$ if and only if $P_\sigma (\bv_1 - \bv_2) = 0$.
  \item\label{thmpropsWvi} For every $\sigma \in \mathbb{R}$,
    \[
      W \circ \tau_\sigma( \bv) = \tau_\sigma \circ W(\bv) \quad \forall \bv \in B_X(\rho).
    \]
 \end{enumerate}
\end{thm}
\begin{proof}
We provide only the proof for a Type II interpolant operator $J$, since the proof for a Type I interpolant follows analogously.

Items \eqref{thmpropsWi}
and \eqref{thmpropsWii} follow directly from the estimates in items \eqref{estw1}-\eqref{estdiffw2} of Proposition \ref{propexistuniqw}. In order to prove \eqref{thmpropsWiii}, let $\bv \in B_X(\rho)$ and $\bu(s)$ for $ s \in \R$ be a solution of \eqref{eqNSE}, and it is a trajectory in the global attractor $\mA$ such that $\|\nabla J\bu(s) - \nabla\bv(s)\|_{L^2} \to 0$, as $s \to \infty$. Denote $\bw = W(\bv)$, $\tw = \bw - \bu$ and $\tv = J\bu - \bv$. Subtracting \eqref{eqNSE} from \eqref{dataeqn}, we obtain
\begin{multline}
\frac{\rd \tw}{\rd s} + \nu A \tw + B(\tw,\bw) + B(\bw,\tw) - B(\tw,\tw) = -\beta \nu \kappa_0^2 P_\sigma (J \bw - \bv) \\
= -\beta \nu \kappa_0^2 \tw - \beta \nu \kappa_0^2 P_\sigma (J\tw - \tw) - \beta \nu \kappa_0^2 P_\sigma \tv.
\end{multline}
Proceeding exactly as in the proof of item \eqref{estdiffw1} of Proposition \ref{propexistuniqw}, we obtain, by using conditions \eqref{condbeta} and \eqref{condbetah}, that for every $\sigma, t \in \mathbb{R}$ with $\sigma < t$,
\be\label{ineqwtildeb}
 \|\nabla\tw (t)\|_{L^2}^2 \leq \|\nabla\tw(\sigma)\|_{L^2}^2 \Exp^{-\frac{\beta\nu\kappa_0^2}{2}(t-\sigma)} + 4 \sup_{\sigma \leq s \leq t} \|\nabla\tv(s)\|_{L^2}^2.
\ee
Taking the $\limsup$ as $t \to \infty$, we have
\[
 \limsup_{t\to \infty} \|\nabla\tw(t)\|_{L^2}^2 \leq 4 \sup_{\sigma \leq s < \infty} \|\nabla\tv(s)\|_{L^2}^2.
\]
Thus, taking the limit as $\sigma \to \infty$ and using that, by hypothesis, $\|\nabla\tv(s)\|_{L^2} \to 0$, as $s \to \infty$, we obtain
\[
 \lim_{t\to \infty} \|\nabla \tw(t)\|_{L^2} = 0,
\]
which proves \eqref{thmpropsWiii}.

In order to prove \eqref{thmpropsWiv}, let $\bu(s)$, $ s \in \R$, be a solution of \eqref{eqNSE},
which is a trajectory  on the global attractor $\mA$, and assume that $J$ satisfies \eqref{h1type2}. Thus, we have
\begin{multline}
 \|\nabla J\bu(s)\|_{L^2} \leq \|\nabla J\bu(s) - \nabla \bu(s)\|_{L^2} + \|\nabla\bu(s)\|_{L^2} \\
 \leq (1 + \tilde{c}_{2,1})\|\nabla\bu(s)\|_{L^2} + \tilde{c}_{2,2} h \|A\bu(s)\|_{L^2} \\
 \leq (1 + \tilde{c}_{2,1}) \nu \kappa_0 G + \tilde{c}_{2,2} h c_2 \nu \kappa_0^2 (G + c_L^{-2})^3 \\
 \leq (1 + \tilde{c}_{2,1}) \nu \kappa_0 G + \tilde{c}_{2,2} c_2 \sqrt{c_2^*} \frac{ \nu \kappa_0}{\beta^{1/2}} (G + c_L^{-2})^3,
\end{multline}
where in the last inequality we used hypothesis \eqref{condbetah} and \eqref{boundmADA}.
Thus,
\[
 \sup_{s\in \mathbb{R}} \frac{\|\nabla J\bu(s)\|_{L^2}}{\nu \kappa_0} \leq \max \{1 + \tilde{c}_{2,1}, \tilde{c}_{2,2} c_2 \sqrt{c_2^*} \} \left[ G + \frac{(G + c_L^{-2})^3}{\beta^{1/2}}\right].
\]
Hence, if $\rho \geq c_3^* [G + \beta^{-1/2} (G + c_L^{-2})^3]$, with $c_3^* = \max \{1 + \tilde{c}_{2,1}, \tilde{c}_{2,2} c_2 \sqrt{c_2^*} \}$, then $\|J\bu\|_X \leq \rho$, i.e., $J\bu \in B_X(\rho)$. From \eqref{ineqwtildeb}, it then follows that
\[
  \|\nabla[W \circ \mJ(\bu) - \bu](t)\|_{L^2}^2 \leq \|\nabla[W \circ \mJ(\bu) - \bu](\sigma)\|_{L^2}^2 \Exp^{-\frac{\beta\nu\kappa_0^2}{2}(t-\sigma)}
\]
Thus, taking the limit as $\sigma \to - \infty$, it follows that $W \circ \mJ (\bu) = \bu$.

Moreover, if $\mJ \circ W (\bv) = \bv$ for some $\bv \in B_X(\rho)$, then, denoting $\bw = W(\bv)$, we have $J\bw(t) -\bv(t) = 0$, for all $t \in \R$. Thus, we see from \eqref{dataeqn} that
$\bw(t)$ for all $t \in \mathbb{R}$, is a solution of \eqref{eqNSE} which is uniformly bounded with respect to the norm in $V$. Thus, $\bw(t) \in \mA$, for all $t \in \mathbb{R}$ (see the characterization of $\mA$ in section 2.1).

In order to prove \eqref{thmpropsWv}, let $\bv_1,\bv_2 \in B_X(\rho)$ and denote $\bw_1 = W(\bv_1)$, $\bw_2 = W(\bv_2)$ and $\tw = \bw_2 - \bw_1$. Then, it follows from \eqref{dataeqn} that
\be\label{eqtildev}
\beta \nu \kappa_0^2 P_\sigma \tv = \frac{\rd \tw}{\rd s} + \nu A \tw + B(\tw,\bw_2) + B(\bw_1, \tw) + \beta \nu \kappa_0^2 P_\sigma J \tw.
\ee

If $W(\bv_1) = W(\bv_2)$, i.e., $\tw \equiv 0$, then it follows from \eqref{eqtildev} and the linearity of $J$ that $P_\sigma \tv(s) = 0$, for a.e. $s \in \mathbb{R}$. But since $\bv_1,\bv_2 \in \Cb(\mathbb{R};J(\dot{H}^1(\Omega))^2)$ then, in fact, $P_\sigma \tv(s) = 0$, for every $s \in \mathbb{R}$.

On the other hand, if $P_\sigma \tv = 0$, then it follows from the uniqueness of solutions to \eqref{dataeqn} that $\tw = 0$, since $\tw=0$ is a solution of \eqref{dataeqn}. This finishes the proof of \eqref{thmpropsWv}.

In order to prove (vi), notice that, given $\bv \in B_X(\rho)$,  $\tau_\sigma \circ W(\bv)$ is a solution of \eqref{dataeqn} corresponding to $\tau_\sigma \bv$. By the uniqueness of solutions to \eqref{dataeqn}, it follows that $\tau_\sigma \circ W(\bv) = W \circ \tau_\sigma (\bv)$. This proves \eqref{thmpropsWvi}.
\end{proof}

\subsection{Fr\'echet differentiability of the determining map $W$}

In this subsection, we show the Fr\'echet differentiability property of the determining map $W$ given in Definition \ref{defW0}, and explicitly identify its derivative.
Although this property is not used in the results of section \ref{secStatSol}, we present it here due to its own importance.

First, we state a result on well-posedness of an auxiliary, linear evolution equation, whose solutions are directly related to the Fr\'echet derivative of $W$. Its proof is similar to the proof of well-posedness of \eqref{dataeqn} (see Proposition \ref{propexistuniqw}), so we omit its details.

\begin{prop}\label{propexistuniqwstar}
 Assume the hypotheses of Proposition \ref{propexistuniqw}. Let $\obv \in X$ and
 $\bw \in C_b(\R; V)$ satisfying the bound \eqref{estwVb}.
 Consider the following evolution equation on $\R $, which is the linearization of \eqref{dataeqn} around $\bw$:
 \be\label{eqwstar}
  \frac{\rd \bw^*}{\rd s} + \nu A \bw^* + B(\bw,\bw^*) + B(\bw^*,\bw) = -\beta \nu \kappa_0^2 P_\sigma (J \bw^* - \obv),
 \ee
 with $\nu > 0$ the same as in \eqref{eqNSE}. Then, equation \eqref{eqwstar} has a unique, bounded solution $\bw^*$ on $\R$ satisfying
 \be \label{freq}
   \bw^* \in \Cb(\mathbb{R}; V) \cap \Lb^2(\mathbb{R};D(A)), \quad \frac{\rd \bw^*}{\rd s} \in \Lloc^2(\mathbb{R};H).
 \ee
 Moreover, $\|\bw^*\|_Y \le 2(3+\beta)^{1/2}\|\obv\|_X$.
\end{prop}

From the result of Proposition \ref{propexistuniqwstar}, we conclude that, given $\rho > 0$, $\beta > 0$ and $h > 0$, satisfying conditions \eqref{condbeta} and \eqref{condbetah}, we can define a mapping $\mD: B_X(\rho) \times X \to Y$ given by
\be\label{defmD}
 \mD(\bv,\obv) = \bw^*,
\ee
where $\bw^*$ is the unique, bounded solution of \eqref{eqwstar}, satisfying \eqref{freq}, and which corresponds to $\bw = W(\bv)$, with $\bv \in B_X(\rho)$, and $\obv \in X$. Indeed, by item \eqref{estw1} of Proposition \ref{propexistuniqw}, it follows that $\bw = W(\bv) \in  C_b(\R;V)$ and it satisfies the bound \eqref{estwVb}.  Moreover, the mapping $\mD(\bv,\cdot):X \to Y$ defines a bounded linear operator from $X$ to $Y$, with operator norm bounded by $2(3+\beta)^{1/2}$.

It turns out that, for a fixed $\bv \in B_X(\rho)$, the bounded linear operator $\mD(\bv,\cdot):X \to Y$ is precisely the Fr\'echet derivative of $W$ at $\bv$, as shown in the next theorem.

\begin{thm}
  Assume the hypotheses of Proposition \ref{propexistuniqw}. Then, the mapping $W: B_X(\rho) \to Y$ is Fr\'echet differentiable and its Fr\'echet derivative at $\bv \in B_X(\rho)$, denoted by $D W (\bv): X \to Y$, is given by
  \be\label{eqFrechetderiv}
   D W (\bv) (\obv) = \mD(\bv,\obv) \quad \forall \obv \in X,
  \ee
  with $\mD(\bv,\obv)$ as defined in \eqref{defmD}.
\end{thm}
\begin{proof}
 Assume $J$ is an interpolant operator satisfying \eqref{type2}. The proof of the case when $J$ satisfies \eqref{type1} follows analogously.

 Let $\bv, \tv \in B_X(\rho)$ and denote $\bw = W(\bv)$, $\tw = W(\tv)$, $\obw = \tw - \bw$ and $\obv = \tv - \bv$.

 In order to prove \eqref{eqFrechetderiv}, we need to show that
 \be\label{eqlittleoh}
  \|\obw - \mD(\bv,\obv)\|_Y = o(\|\obv\|_X).
 \ee

 Notice that $\obw$ satisfies
 \be\label{eqobw}
  \frac{\rd \obw}{\rd s} + \nu A \obw + B(\bw, \obw) + B(\obw,\bw) + B(\obw,\obw) = -\beta \nu \kappa_0^2 P_\sigma (J\obw - \obv).
 \ee

 Denote $\bw^* = \mD(\bv,\obv)$. Subtracting \eqref{eqwstar} from \eqref{eqobw}, we obtain that
 \begin{multline}\label{eqdiffobwbwstar}
  \frac{\rd }{\rd s}(\obw - \bw^*) + \nu A (\obw - \bw^*) + B(\bw,\obw - \bw^*) + B(\obw - \bw^*,\bw) + B(\obw,\obw) \\
  = -\beta \nu \kappa_0^2 (\obw - \bw^*) - \beta \nu \kappa_0^2 P_\sigma [J(\obw - \bw^*) - (\obw - \bw^*)].
 \end{multline}

 Taking the inner product in $H$ of \eqref{eqdiffobwbwstar} with $A(\obw - \bw^*)$ and applying Cauchy-Schwarz, Young's inequality and property \eqref{h1type2} of $J$, we obtain that
 \begin{multline}
  \frac{1}{2} \frac{\rd }{\rd s} \|\nabla\obw - \nabla\bw^*\|_{L^2}^2 + \nu \|A(\obw - \bw^*)\|_{L^2}^2 \\
  \leq 2 \max \left\{ |( B(\bw,\obw - \bw^*), A(\obw - \bw^*))|, |( B(\obw - \bw^*,\bw), A(\obw - \bw^*))| \right\}\\
  + \frac{\nu}{8} \|A(\obw - \bw^*)\|_{L^2}^2 + \frac{2}{\nu} \|B(\obw,\obw)\|_{L^2}^2 - \frac{\beta \nu \kappa_0^2}{2} \|\nabla \obw - \nabla \bw^*\|_{L^2}^2 \\
  + \beta \nu \kappa_0^2 \left(\frac{c_{2,1}^2}{2} + c_{2,2}\right) h^2 \|A(\obw - \bw^*)\|_{L^2}^2
 \end{multline}

 Now, proceeding analogously as in \eqref{nonlinLip}-\eqref{wtildeineq2} and using conditions \eqref{condbeta} and \eqref{condbetah} on $\beta$ and $h$, it follows that
 \be\label{eqdiffobwbwstar2}
  \frac{\rd}{\rd s}\|\nabla \obw - \nabla \bw^*\|_{L^2}^2 + \frac{\nu}{2} \|A(\obw - \bw)\|_{L^2}^2 + \frac{\beta \nu \kappa_0^2}{2} \|\nabla \obw - \nabla\bw^*\|_{L^2}^2 \leq \frac{4}{\nu} \|B(\obw,\obw)\|_{L^2}^2.
 \ee

 From \eqref{estNonlTermC} and Poincar\'e inequality \eqref{ineqPoincare}, we have
 \be\label{estBobwobw}
  \|B(\obw,\obw)\|_{L^2}^2 \leq \frac{c_L^4}{\kappa_0} \|\nabla\obw\|_{L^2}^3 \|A\obw\|_{L^2}.
 \ee

 Thus, integrating \eqref{eqdiffobwbwstar2} on the interval $[\sigma,s]$ and using \eqref{estBobwobw}, yields
 \begin{multline}\label{estdiffobwbwstar}
  \|\nabla (\obw - \bw^*)(s)\|_{L^2}^2 \leq \Exp^{-\frac{\beta \nu \kappa_0^2}{2} (s-\sigma)} \|\nabla(\obw - \bw^*)(\sigma)\|_{L^2}^2 \\
  + \frac{4c_L^4}{\nu \kappa_0} \int_{\sigma}^s \Exp^{-\frac{\beta \nu \kappa_0^2}{2} (s-\tau)} \|\nabla\obw(\tau)\|_{L^2}^3 \|A\obw(\tau)\|_{L^2} \rd \tau.
 \end{multline}

 Now, we choose $\sigma = s - \frac{n}{\nu \kappa_0^2}$, for some $n \in \mathbb{N}$. Notice that
 \begin{multline}\label{estintAobw}
  \int_{\sigma}^s \|A\obw(\tau)\|_{L^2} \Exp^{-\frac{\beta \nu \kappa_0^2}{2} (s-\tau)} \rd \tau = \sum_{j = 1}^n \int_{s - \frac{j}{\nu\kappa_0^2}}^{s - \frac{(j-1)}{\nu\kappa_0^2}} \|A\obw(\tau)\|_{L^2} \Exp^{-\frac{\beta \nu \kappa_0^2}{2} (s-\tau)} \rd \tau \\
  \leq \sum_{j = 1}^n \left( \int_{s - \frac{j}{\nu\kappa_0^2}}^{s - \frac{(j-1)}{\nu\kappa_0^2}} \|A\obw(\tau)\|_{L^2}^2 \right)^{1/2} \left( \frac{\Exp^{-\beta(j-1)} - \Exp^{-\beta j}}{\beta \nu \kappa_0^2} \right)^{1/2} \\
  \leq \sup_{t\in \R} \left( \frac{1}{\nu \kappa_0^2} \int_t^{t + \frac{1}{\nu \kappa_0^2}} \|A\obw(\tau)\|_{L^2}^2 \rd \tau \right)^{1/2} \frac{(1 - \Exp^{-\beta})^{1/2}}{\beta^{1/2}} \sum_{j = 1}^n \Exp^{-\frac{\beta(j-1)}{2}}\\
  \leq \|\obw\|_Y C_\beta (1 - \Exp^{\frac{-\beta n}{2}}),
 \end{multline}
where $C_\beta = (1 - \Exp^{-\beta})^{1/2}[\beta^{1/2}(1 - \Exp^{-\beta/2})]^{-1}$.

Hence, from \eqref{estdiffobwbwstar} and \eqref{estintAobw}, we obtain that
\begin{multline}
 \|\nabla(\obw - \bw^*)(s)\|_{L^2}^2 \leq \Exp^{-\frac{\beta \nu \kappa_0^2}{2} (s-\sigma)} \|\nabla(\obw - \bw^*)(\sigma)\|_{L^2}^2\\
 + 4 c_L^4 (\nu \kappa_0)^2 \sup_{t \in \R} \frac{\|\nabla\obw(\tau)\|_{L^2}^3}{(\nu \kappa_0)^3} \int_{\sigma}^s \|A\obw(\tau)\|_{L^2} \Exp^{-\frac{\beta \nu \kappa_0^2}{2} (s-\tau)} \rd \tau \\
 \leq \Exp^{-\frac{\beta \nu \kappa_0^2}{2} (s-\sigma)} \|\nabla(\obw - \bw^*)(\sigma)\|_{L^2}^2  + 4 c_L^4 C_\beta (\nu \kappa_0)^2 \|\obw\|_Y^4 (1 - \Exp^{-\frac{\beta n}{2}}).
\end{multline}

Then, taking the limit as $\sigma \to -\infty$ (i.e., $n \to \infty$) and using that $\obw, \bw^* \in \Cb(\R;V)$, it follows that
\be\label{estsupsdiffobwbwstarYa}
 \sup_{s\in \R} \frac{\|\nabla(\obw - \bw^*)(s)\|_{L^2}^2}{(\nu \kappa_0)^2} \leq 4 c_L^4 C_\beta \|\obw\|_Y^4 \leq c c_L^4 C_\beta (3 + \beta)^2 \|\obv\|_X^4,
\ee
where, in the last inequality, we used item \eqref{thmpropsWii} of Theorem \ref{thmpropsW}.

Now, integrating \eqref{eqdiffobwbwstar2} on the interval $[s,s + (\nu \kappa_0^2)^{-1}]$ and using \eqref{estBobwobw} and \eqref{estsupsdiffobwbwstarYa}, yields
\begin{multline}\label{estsupsdiffobwbwstarYb}
 \frac{1}{\nu \kappa_0^2} \int_s^{s + \frac{1}{\nu \kappa_0^2}} \|A(\obw - \bw^*)(\tau)\|_{L^2}^2 \rd \tau \leq c c_L^4 C_\beta (3 + \beta)^2 \|\obv\|_X^4 \\
 + c c_L^4 \left( \sup_{\tau \in \R} \frac{\|\nabla\obw(\tau)\|_{L^2}^3}{(\nu \kappa_0)^3}\right) \int_s^{s + \frac{1}{\nu \kappa_0^2}} \|A\obw(\tau)\|_{L^2}\rd \tau \\
 \leq c c_L^4 C_\beta (3 + \beta)^2 \|\obv\|_X^4 + c c_L^4 \|\obw\|_Y^3 \left( \frac{1}{\nu \kappa_0^2} \int_s^{s + \frac{1}{\nu \kappa_0^2}} \|A\obw(\tau)\|_{L^2}^2 \rd \tau \right)^{1/2} \\
 \leq c c_L^4 (1 + C_\beta) (3 + \beta)^2 \|\obv\|_X^4.
\end{multline}

Hence, from \eqref{estsupsdiffobwbwstarYa} and \eqref{estsupsdiffobwbwstarYb}, we conclude that
\[
 \|\obw - \bw^* \|_Y \leq \widetilde{C}_\beta \|\obv\|_X^2,
\]
where $\widetilde{C}_\beta = [c c_L^4 (1 + 2 C_\beta) (3 + \beta)^2]^{1/2}$. This proves \eqref{eqlittleoh} and concludes the proof of the theorem.
\end{proof}

\subsection{The map $W_+$}

We now restrict our attention to functions defined on $\R_+ = [0,\infty)$ and introduce an analogous framework to the one developed in Subsection \ref{subsecDetMapW}. Thus, we consider the Banach spaces
\[
 Y_+ = \Cb(\R_+;V) \cap \Lb^2 (\R_+;D(A))
\]
and
\[
 X_+ = \Cb(\R_+; (\dot{H}^1(\Omega))^2)
\]
endowed with the norms
\[
 \|\bu\|_{Y_+} = \left\{ \sup_{s\geq 0} \frac{\|\nabla\bu(s)\|_{L^2}^2}{(\nu \kappa_0)^2} + \sup_{s\geq 0} \frac{1}{\nu \kappa_0^2} \int_s^{s + \frac{1}{\nu\kappa_0^2}} \|A\bu(r)\|_{L^2}^2 \rd r \right\}^{1/2}
\]
and
\[
 \|\bv\|_{X_+} = \sup_{s\geq 0} \frac{\|\nabla\bv(s)\|_{L^2}}{\nu \kappa_0},
\]
respectively.

Now, given $\bv \in X_+$, we consider the following initial-value problem:
\be\label{dataeqn+}
 \frac{\rd \bw}{\rd s} + \nu A \bw + B(\bw,\bw) = \f - \beta \nu \kappa_0^2 \Ps (J\bw - \bv), \quad s \in \mathbb{R}_+,
\ee
\be\label{initcond+}
\bw(0) = 0,
\ee
where $\nu$ and $\f$ are the same as in \eqref{eqNSE}.

Similarly as in Proposition \ref{propexistuniqw}, one can show that system \eqref{dataeqn+}-\eqref{initcond+} is well-posed. In fact, all the results given in Proposition \ref{propexistuniqw} are still valid after replacing $\R$ by $\R_+$ and $X$ by $X_+$.

Therefore, given $\rho>0$, $\beta > 0$ and $h > 0$ satisfying conditions \eqref{condbeta} and \eqref{condbetah}, we can define a mapping $W_+ : B_{X_+}(\rho) \to Y_+$ given by
\be\label{defW+}
W_+(\bv) = \bw,
\ee
where $\bw$ is the unique solution of \eqref{dataeqn+}-\eqref{initcond+} corresponding to $\bv \in B_{X_+}(\rho)$.

Moreover, we have the following analogous version of Theorem \ref{thmpropsW} for
$W_+$:

\begin{thm}\label{thmpropsW+}
 Assume the hypotheses of Proposition \ref{propexistuniqw}, with $X$ and $Y$ replaced by
 $X_+$ and $Y_+$ respectively.
 Then, the mapping $W_+: B_{X_+}(\rho) \to Y_+$ defined in \eqref{defW+} satisfies the following properties:
 \begin{enumerate}[(i)]
  \item\label{thmpropsW+i} For every $\bv \in B_{X_+}(\rho)$, $W_+(\bv) \in B_{Y_+}(\sqrt{M})$, with
    \[
      M = M(G,\rho) = 2(2+\beta)\left( \frac{G^2}{\beta} + \rho^2 \right).
    \]
  \item\label{thmpropsW+ii} $W_+$ is a Lipschitz mapping from $B_{X_+}(\rho)$ to $Y_+$ with Lipschitz constant $2(3+\beta)^{1/2}$.
  \item\label{thmpropsW+iii} Let $\bv \in B_{X_+}(\rho)$ and $\bu$ be a solution of \eqref{eqNSE} with $\bu(0) = \bu_0 \in V$, and assume that $\|\nabla J\bu(s) - \nabla\bv(s)\|_{L^2} \to 0$, as $s \to \infty$. Then $\|\nabla[W_+(\bv) - \bu](s)\|_{L^2} \to 0$, as $s \to \infty$, uniformly with respect to $\bu_0$ in any bounded set in $V$.
  \item\label{thmpropsW+iiip} Let $\bu$ be a solution of \eqref{eqNSE} with $\bu(0) = \bu_0 \in V$. Then, $\|\nabla[W_+(J \bu) - \bu](s)\|_{L^2} \to 0$ exponentially, as $s \to \infty$.
  \item\label{thmpropsW+iv} Let $\bv_1,\bv_2 \in B_{X_+}(\rho)$. Then, $W_+(\bv_1) = W_+(\bv_2)$ if and only if $P_\sigma (\bv_1 - \bv_2) = 0$.
  \item\label{thmpropsW+v} For every $\sigma \geq 0$,
    \[
      W_+ \circ \tau_\sigma (\bv) = \tau_\sigma \circ W_+(\bv) \quad \forall \bv \in B_{X_+}(\rho).
    \]
 \end{enumerate}
\end{thm}

The proof of items \eqref{thmpropsW+i}-\eqref{thmpropsW+iii} and \eqref{thmpropsW+iv}-\eqref{thmpropsW+v} is similar to the proof of Theorem \ref{thmpropsW}. The proof of item \eqref{thmpropsW+iiip} follows from \eqref{ineqwtildeb}, by noting that $\tilde{\bv}(s) = 0$, for all $s \in \R_+$.

\begin{rmk}
Notice that, \emph{a priori}, there is no relation between the maps $W$ and $W_+$ defined on the spaces $X$ and $X_+$, except the fact that they are defined following a similar approach. The construction of the map $W$ on $X$ given here is inspired by \cite{FoiasJollyKravchenkoTiti2014}, while the construction of $W_+$ on $X_+$ uses well-established tools from \cite{AzouaniOlsonTiti2014}. However, in principle, they are not much different, except that for $W$, one has to deal with an evolution equation on all of $\R$, and thus without an initial value.
\end{rmk}

\section{Study of Statistical Solutions}\label{secStatSol}

In this section, we use the determining maps $W$ and $W_+$ introduced in section \ref{secDetMap} to obtain results concerning statistical solutions of the Navier-Stokes equations. In subsections \ref{subsecPrelimMeasTh}-\ref{subsecStatSol}, we recall the preliminary material needed, while Subsections \ref{subsecEnsDA} and \ref{subsecDetParamStatSol} contain the main results.

\subsection{Preliminaries on Measure Theory}\label{subsecPrelimMeasTh}

Let $(\mM_1, \Sigma_1)$ and $(\mM_2, \Sigma_2)$ be measure spaces and $T: \mM_1 \to \mM_2$ be a measurable map. Then, given a measure $\mu$ on $\mM_1$, the \emph{push-forward measure} of $\mu$ by $T$, denoted $T\mu$, is defined as
\be\label{defpushforwardmeas}
 T\mu(E) = \mu (T^{-1} E) \quad \forall E \in \Sigma_2.
\ee

Moreover, given a measure space $(\mM, \Sigma)$ and a measure $\mu$ on $\mM$, for every $\widetilde{\mM} \subset \mM$ with $\widetilde{\mM} \in \Sigma$, we denote by $\mu|_{\widetilde{\mM}}$ the restriction of $\mu$ to $(\widetilde{\mM}, \widetilde{\Sigma})$, where
\be\label{restrictedsigmaalg}
 \widetilde{\Sigma} = \{E \cap \widetilde{\mM} \,:\, E \in \Sigma \}.
\ee

When $\mM$ is a topological space, we denote by $\mP(\mM)$ the space of Borel probability measures on $\mM$, i.e., the space of measures $\mu$ defined on the sigma-algebra $\Sigma$ of Borel subsets of $\mM$.

We say that a measure $\mu \in \mP(\mM)$ is \emph{carried} by $E \in \Sigma$ if $\mu(E) = 1$.

A measure $\mu \in \mP(\mM)$ is said to be \emph{tight} if for every $E  \in \Sigma$,
\[
 \mu(E) = \sup \{ \mu(\mK) \,:\, \mK \mbox{ is a compact subset of }\mM \mbox{ and } \mK \subset E\}.
\]

Moreover, in case $\mM$ is a metric space with a metric $d$, we denote by $\mP_1(\mM,d)$ the subset of all measures $\mu\in \mP(\mM)$ satisfying
\[
 \int_{\mM} d(x,y) \rd \mu(x) < \infty \quad \forall y \in \mM.
 \]

Notably, if $\mu \in \mP(\mM)$ is carried by a bounded subset of $(\mM,d)$, then, clearly, $\mu \in \mP_1(\mM,d)$.

Let $\Lip(\mM,d)$ denote the space of real-valued Lipschitz continuous functions on a metric space $(\mM,d)$, endowed with the seminorm
\[
 \|\varphi\|_{\textnormal{Lip}} := \sup_{x,y \in \mM, x\neq y} \frac{|\varphi(x) - \varphi(y)|}{d(x,y)}.
\]
We recall the definition of the Kantorovich metric in $\mP_1(\mM,d)$ (see, e.g., \cite{Dudley2002}), given by
\begin{multline}
\gamma_\mM (\mu, \eta) = \\
\sup \left\{ \left| \int_\mM \varphi(x) \rd \mu(x) - \int_{\mM} \varphi(x) \rd \eta (x) \right| \,:\, \varphi \in \Lip(\mM,d), \,\, \|\varphi\|_{\textnormal{Lip}} \leq 1 \right\} \\
\quad \forall \mu, \eta \in \mP_1(\mM,d).
\end{multline}

In the sequel, we denote by $\gamma_H$ the Kantorovich metric on $\mP_1(H,d_H)$, where $d_H$ is the metric induced by the norm $\|\cdot \|_{L^2}$, i.e., $d_H (\bu,\bv) = \|\bu - \bv\|_{L^2}$ for all $\bu,\bv \in H$. Also, we denote by $\Gamma_H$ the Kantorovich metric on $\mP_1(\Cloc(\R_+,H),d_0^+)$, where $d_0^+$ is a suitable metric on $\Cloc(\R_+,H)$ defined in subsection \ref{subsecMetrics}, below.

\subsection{Metrics on spaces of continuous functions}\label{subsecMetrics}

For a given Banach space $Z$,
a useful topology on the space of continuous functions $C(I,Z)$ is the \emph{topology of uniform convergence on compact subsets} (see, e.g., \cite{Kelley1975}), which is defined as the topology generated by the sub-base of neighborhoods of the form
\[
 \mN(K,O) =\{ u \in C(I,Z) \,:\, u(t) \in O \;\; \forall t \in K \},
\]
where $K \subset I$ is a compact subset and $O \subset Z$ is an open neighborhood of the origin. Notably, a sequence $\{u_n\}_{n\in \mathbb{N}}$ in $C(I,Z)$ converges to $u \in C(I,Z)$ with respect to this topology if, and only if, for every compact subset $K \subset I$,
\[
 \sup_{t \in K} \|u_n(t) - u(t)\|_Z \to 0 \mbox{ as } n \to \infty,
\]
where $\|\cdot\|_Z$ denotes the norm in $Z$.
From now on, we use the notation $\Cloc(I,Z)$ to denote the space $C(I,Z)$ endowed with the topology of uniform convergence on compact subsets.

Consider the spaces $\Cloc(I,$ $(\dot{L}^2(\Omega))^2)$ and $\Cloc(I,(\dot{H}^1(\Omega))^2)$ of continuous functions on an interval $I \subset \R$ with values in $(\dot{L}^2(\Omega))^2$ and $(\dot{H}^1(\Omega))^2$, respectively, endowed with the corresponding topology of uniform convergence on compact subsets,
 as defined above. 
The fact that $(\dot{L}^2(\Omega))^2$ and $(\dot{H}^1(\Omega))^2$ are, in particular, metric spaces implies that $\Cloc(I,(\dot{L}^2(\Omega))^2)$ and $\Cloc(I,(\dot{H}^1(\Omega))^2)$ are metrizable. Indeed, let $\{K_n\}_{n \geq 1}$ be a sequence of compact subintervals of $I$ such that $I = \bigcup_{n \geq 1} K_n$. Then, a compatible metric with the topology of $\Cloc(I,(\dot{L}^2(\Omega))^2)$ is given by
\be\label{defd0}
d_0^I(\bu,\bv) = \sum_{n \geq 1} \frac{1}{2^n} \frac{\sup_{t \in K_n} \|\bu(t) - \bv(t)\|_{L^2}}{\nu + \sup_{t \in K_n} \|\bu(t) - \bv(t)\|_{L^2}} \quad \forall \bu, \bv \in \Cloc(I,(\dot{L}^2(\Omega))^2),
\ee
while a compatible metric with the topology of $\Cloc(I,(\dot{H}^1(\Omega))^2)$ is given by
\begin{multline}\label{defd1}
d_1^I(\bu,\bv) = \sum_{n \geq 1} \frac{1}{2^n} \frac{\sup_{t \in K_n} \|\nabla\bu(t) - \nabla\bv(t)\|_{L^2}}{\nu \kappa_0 + \sup_{t \in K_n} \|\nabla\bu(t) - \nabla\bv(t)\|_{L^2}}\\
 \forall \bu, \bv \in \Cloc(I,(\dot{H}^1(\Omega))^2).
\end{multline}

 In particular, when $I = \R$, we consider $K_n = [-n (\nu \kappa_0^2)^{-1},n (\nu \kappa_0^2)^{-1}]$ for all $n \geq 1$; and when $I = \R_+$, we consider $K_n^+ = [0,n (\nu \kappa_0^2)^{-1}]$ for all $n \geq 1$. Moreover, in order to simplify the notation, we denote $d_0 = d_0^{\R}$, $d_0^+ = d_0^{\R_+}$, $d_1 = d_1^{\R}$ and $d_1^+ = d_1^{\R_+}$.

\subsection{Statistical solutions}\label{subsecStatSol}

Statistical solutions of an evolution equation are given as probability measures which represent the evolution in time of probability distributions of the model state variables according to the underlying dynamics. They are particularly useful in the study of a physical system for which there is uncertainty with respect to the initial condition, and one would like to determine how this uncertainty is going to evolve in time.

Such statistical solutions can be of two types: first,  the  {\it phase space statistical solutions}, i.e., a family of probability measures $\{\mu_t\}_{t\in I}$ defined on the phase space and indexed by the time variable $t$, representing the evolution in time of probability distributions of the state variables in the phase space; and secondly,
 the  {\it trajectory statistical solutions}, i.e., a single space-time probability measure $\mu$ defined on the space of trajectories, whose support is contained in the set of all possible individual solutions of the evolution equation.

These concepts proved to be the natural rigorous mathematical framework for investigating the statistical properties of 3D turbulent flows as demonstrated in the pioneering works of Foias and Prodi in \cite{Foias72, Foias73}, concerning phase-space statistical solutions, and later by Vishik and Fursikov in \cite{VF78, VF88}, concerning trajectory statistical solutions. More recently, in \cite{FRT2010, FRT2013}, inspired by the definition of a trajectory statistical solution given in \cite{VF78, VF88}, the authors provide a slightly different definition, still in the context of the 3D NSE, which allows for a connection with the definition of a phase-space statistical solution given in \cite{Foias72, Foias73}. An extension of this more recent definition to an abstract framework that can be applied to a large class of evolution equations was given in \cite{BronziMondainiRosa2016}.

When the evolution equation is well-posed, as is the case of the 2D NSE, then, given an initial probability measure $\mu_0$ defined on the associated phase space, obtaining its evolution in time is rather simple. Indeed, using the notation related to the 2D NSE from subsection \ref{subsecNSE}, for the first type of statistical solution, one considers the family of measures given as the push-forward measures of $\mu_0$ by the solution semigroup $S(t)$, $t \geq 0$, i.e., $\{S(t)\mu_0\}_{t\geq 0}$. Moreover, for the second type of statistical solution, let $\mS$ be the solution operator associated to \eqref{eqNSE}, i.e.
\begin{eqnarray*}
\qquad \mS:H &\rightarrow & C(\R_+,H)\\
                \bu_0 &\mapsto& \bu(t)\ \mbox{for all}\ t \in \R_+.
\end{eqnarray*}
where $\bu(t)$ for $t \in \R_+$ is the unique solution of \eqref{eqNSE} satisfying \eqref{propssolNSE} with $\bu(0) = \bu_0$.  From the well-known stability estimates, i.e., Lipschitz continuous dependence on the initial data,
for the 2D NSE \cite{bookcf1988, Temambook1995, Temambook2001}, it follows  that $\mS$ is a continuous, and therefore Borel measurable, map from $H$ to $\Cloc(\R_+, H)$.
Then, the push-forward of $\mu_0$ by $\mS$, i.e., $\mS \mu_0$, is a measure defined on the space of trajectories which is carried by the set of solutions of \eqref{eqNSE}.


Now, we provide the definitions of statistical solutions on the space of trajectories and on the phase space for the 2D NSE, on any time interval $I \subset \R$. First, for every interval $I \subset \R$, let us denote
\be\label{defTI}
\fT^I = \{\bu \in \Cloc(I;H) \,:\, \bu \text{ is a solution of \eqref{eqNSE} on } I\}.
\ee

In the next proposition, we show that $\fT^I$ is a measurable subset of $\Cloc(I,H)$.

\begin{prop}\label{propfTIBorel}
 For every interval $I \subset \R$, $\fT^{I}$ is a Borel subset of $\Cloc(I,H)$.
\end{prop}
\begin{proof}
Let $\{\bu_n\} \subset \fT^I$ such that $\bu_n \ra \bu$ in $\Cloc(I; H)$. Obviously,
$I = \bigcup_{j \in \N} [a_j,b_j]$ such that $a_{j+1} \le a_j$ and $b_{j+1} \ge b_j$ for all $j \in \N$. Since $\bu_n(a_j) \ra \bu(a_j)$ in $H$ as $n \ra \infty$, it follows from the continuous dependence on initial data of the solution (cf. \cite{bookcf1988}) that
\[
\sup_{t \in K_j} \|S(t-a_j)\bu_n(a_j)-S(t-a_j)\bu(a_j)\|_{L^2} \ra 0\ \mbox{as}\ n \ra \infty.
\]
Since $S(t-a_j)\bu_n(a_j)=\bu_n(t)$, it follows that $\bu(t)=S(t-a_j)\bu(a_j)$, which is a solution on $K_j$. Thus, $\bu(t)$ is a solution on every $K_j$, and therefore on $I$.
Consequently, $\bu(t) \in \fT^I$. Therefore, $\fT^I$ is closed, and thus a Borel subset of $\Cloc(I,H)$.
\end{proof}



\begin{defs}\label{defVFmeas}
 Given an interval $I \subset \mathbb{R}$, we say that a Borel probability measure $\mu$ on $\Cloc(I,H)$ is a \textbf{trajectory statistical solution of the 2D Navier-Stokes equations \eqref{eqNSE} over $I$} if $\mu$ is carried by $\fT^I$.
\end{defs}

\begin{rmk}
 The abstract definition of a trajectory statistical solution given in \cite{BronziMondainiRosa2016} requires $\mu$ to be a tight measure  as well. However, since every finite Borel measure on a Polish space (i.e., a separable and completely metrizable topological space) is tight (\cite[Theorem 12.7]{AliprantisBorderbook}), and $\Cloc(I;H)$ is a Polish space (see, e.g., \cite[Lemma 3.99]{AliprantisBorderbook}),
 then the tightness condition on $\mu$ in our case is automatically satisfied. Moreover, the definition in \cite{BronziMondainiRosa2016} only requires $\mu$ to be carried by a Borel set containing the set of solutions, i.e., $\fT^I$ in our case. This is done in order to allow for the application to evolution equations for which one cannot determine if its corresponding set of solutions is a Borel set. But since, as shown in Proposition \ref{propfTIBorel}, $\fT^I$ is a Borel set, we can define $\mu$ as being carried by $\fT^I$ directly.
\end{rmk}

Before providing the definition of statistical solutions on phase space, we need to recall the definition of a special class of test functions, called \emph{cylindrical test functions} (see, e.g., \cite{FMRT2001, FRT2013}). These are functions $\Phi: V' \to \R$ of the form
\[
 \Phi(\bu) = \sum_{j=1}^k \phi(\langle \bu, \bv_1 \rangle_{V',V}, \ldots, \langle \bu, \bv_k\rangle_{V',V}) \quad \forall \bu \in V',
\]
where $\phi$ is a continuously differentiable real-valued function on $\R^k$ with compact support and $\bv_1, \ldots, \bv_k \in V$. We denote by $\Phi': V'\to V$ its Fr\'echet derivative, given by
\[
 \Phi'(\bu) = \sum_{j=1}^k \partial_j \phi (\langle \bu, \bv_1 \rangle_{V',V}, \ldots, \langle \bu, \bv_k\rangle_{V',V}) \bv_j \quad \forall \bu \in V',
\]
where $\partial_j \phi$ denotes the derivative of $\phi$ with respect to its $j$-th coordinate.

\begin{defs}\label{defstatsolphasespace}
 Given an interval $I \subset \R$, we say that a family $\{\mu_t\}_{t \in I}$ of Borel probability measures on $H$ is a \textbf{phase-space statistical solution of the 2D Navier-Stokes equations \eqref{eqNSE} over $I$} if
 \begin{enumerate}[(i)]
  \item The function
   \[
     t \mapsto \int_H \varphi(\bu) \rd \mu_t(\bu)
   \]
   is continuous on $I$, for every $\varphi \in \Cb(H)$;
  \item For almost every $t \in I$, the measure $\mu_t$ is carried by $V$ and the function
   \[
    \bu \mapsto \langle \f - \nu A \bu - B(\bu ,\bu ), \bv \rangle_{V',V}
   \]
   is $\mu_t$-integrable, for every $\bv \in V$. Moreover, the map
   \[
    t \mapsto \int_H \langle \f - \nu A \bu - B(\bu,\bu ), \bv \rangle_{V',V} \rd \mu_t(\bu)
   \]
   belongs to $\Lloc^1(I)$, for every $\bv \in V$.
  \item For any cylindrical test function $\Phi$ in $V'$, it holds
   \begin{multline}
   \int_H \Phi(\bu) \rd \mu_t(\bu) = \int_H \Phi(\bu) \rd \mu_{t'}(\bu) \\
    + \int_{t'}^t \int_H \langle \f - \nu A \bu - B(\bu ,\bu ), \bv \rangle_{V',V} \rd \mu_s(\bu) \rd s,
   \end{multline}
   for all $t, t' \in I$  with $t'<t$.
 \end{enumerate}
\end{defs}

In \cite{BronziMondainiRosa2016}, it is shown that, given a trajectory statistical solution $\mu$ on $I \subset \R$, in the sense of Definition \ref{defVFmeas}, the family of measures obtained as its projections in time, i.e., $\{\mE_t \mu\}_{t\in I}$, is a phase-space statistical solution on $I$, in the sense of Definition \ref{defstatsolphasespace}.

Notice that, given $\bu \in \fT^I$, the Dirac measure concentrated on $\bu$, denoted $\delta_\bu$, is clearly a trajectory statistical solution on $I$. Thus, the family of measures $\{\mE_t \delta_{\bu}\}_{t \in I} = \{\delta_{\bu(t)}\}_{t \in I}$ is a phase-space statistical solution on $I$. In fact, the same holds for any convex combination of Dirac measures $\delta_{\bu_1}, \ldots, \delta_{\bu_k}$, with $\bu_1, \ldots, \bu_k \in \fT^I$ and $k \in \mathbb{N}$.  The following proposition elucidates the connection between phase space statistical solutions and trajectory statistical solutions for the 2D NSE.
\begin{prop}
Let $\mu_0$ be a Borel probability measure on $H$ with
\[
 \int_H \|\bu\|_{L^2}^2 \rd \mu_0(\bu) < \infty.
\]
Then the family of measures $\{\mu_t=\mE_t\mS\mu_0=S(t)\mu_0\}_{t \ge 0}$ is the unique
phase space statistical solution with the initial measure given by $\mu_0$ and moreover,
$\mu = \mS\mu_0$ is the unique trajectory statistical solution satisfying $\mE_0\mu=\mu_0$.
\end{prop}

\begin{proof}
The fact that the family of measures $\{\mE_t \mS \mu_0\}_{t \geq 0} = \{S(t) \mu_0\}_{t\geq 0}$ is a phase-space statistical solution on $[0,\infty)$ corresponding to the initial measure $\mu_0$ follows from the definition of a phase space space statistical solution and that of the push forward measure $S(t)\mu_0$, while the uniqueness follows from  \cite[Theorem V.1.4]{FMRT2001}.

For the second part of the statement, clearly, $\mu = \mS \mu_0$ is a trajectory statistical solution on $\Cloc(\R_+,H)$ satisfying the initial condition $\mE_0 \mu = \mu_0$. To show uniqueness, if $\rho$ is any trajectory statistical solution on $\Cloc(\R_+,H)$ satisfying $\mE_0 \rho = \mu_0$, then, for every Borel subset $E$ of $\Cloc(\R_+,H)$, using the fact
that $\mS \circ \mE_0$ is the identity map on $\fT^I$, where $I=\R_+$,
we have
\begin{multline}
 \rho(E) = \rho(E \cap \fT^I) = \rho((\mS\circ \mE_0)^{-1}(E\cap \fT^I)) = \mE_0 \rho (\mS^{-1}(E \cap \fT^I)) \\
 = \mu_0 (\mS^{-1}(E \cap \fT^I)) = \mu_0(\mS^{-1}( E)) = \mS \mu_0(E).
\end{multline}

\end{proof}


\subsection{Ensemble Downscaling Data Assimilation}\label{subsecEnsDA}

In this subsection, our goal is to show that if $\mu$ is a trajectory statistical solution on
$ \R_+$, then the translations (or evaluations) in time of the measure $(W_+ \circ \mJ) \mu$, with $\mJ$ as defined in \eqref{defmJ}, converge to the translations (or evaluations) in time of $\mu$ asymptotically in time, in a suitable sense. In practice, $\mJ \mu$ is a measure constructed through uncertainties associated with measurements of the model state variables and the purpose of $W_+$ is to downscale $\mJ \mu$, and hence
reduce these uncertainties by decreasing the error due to the coarse spatial resolution of the measurements.

Let us consider the following set of trajectories, for every interval $I \subset \R$:
\bea\label{defTB}
\lefteqn{\fTB^I = \{ \bu \in \fT^{I} \,:}\nn \\
&  \qquad \|\nabla\bu(t)\|_{L^2} \leq \sqrt{2} \nu \kappa_0 G, \,\, \|A\bu(t)\|_{L^2} \leq c_2 \nu \kappa_0^2 (G + c_L^{-2})^3 \,\,\, \forall t \in I \},
\eea
where $c_2$ is the constant from \eqref{boundmADA}. Notice that $\fTB^+ = \fTB^{\R_+}$ is a nonempty set, since the global attractor $\mA$ is nonempty, and for every trajectory $\bu(t), t \in \R$ in the global attractor, its restriction to $I$ belongs to $\fT_b^+$,
 due to the estimates \eqref{boundmAH1} and  \eqref{boundmADA}.


As shown in Proposition \ref{propfTIBorel}, $\fTB^I$ is closed, and therefore a Borel subset, of  $\Cloc(I;H)$.
Additionally, as we show in the proposition below,
it is a compact subset of $\Cloc(I;V)$, and therefore it is a compact subset of
$\Cloc(I;H)$ as well.

\begin{lem}\label{lemTBIcptCIV}
 For every interval $I \subset \R$, $\fTB^I$ is a compact subset of $\Cloc(I;V)$.
\end{lem}
\begin{proof}
 First, note that the inclusion $\fTB^I \subset \Cloc(I;V)$ follows from the bounds in \eqref{defTB}, the fact that $\fTB^I \subset \Cloc(I;H)$ and the interpolation inequality
 \[
  \|\nabla (\bu(t_1) - \bu(t_2))\|_{L^2} \leq \|\bu(t_1) - \bu(t_2)\|_{L^2}^{1/2} |A\bu(t_1) - A\bu(t_2)|_{L^2}^{1/2}, \quad \forall t_1,t_2 \in I.
 \]

 Since $\Cloc(I;V)$ is metrizable, it suffices to show that $\fTB^I$ is sequentially compact.
  Note first that due to \eqref{ineqLadyzhenskaya}, for all $\bu \in \fT^I_b$, we have
 \[
 |B(\bu,\bu)| \le \|\bu\|_{L^2}^{1/2}\|\nabla \bu\|_{L^2} \|A\bu\|_{L^2}^{1/2}.
 \]
 From \eqref{boundfTBV}, \eqref{boundfTBDA} and \eqref{nse}, it now follows that
 \be \label{arzela}
 \sup_{\bu \in \fTB^I} \left\|\dt \bu (t)\right\|_{L^2} < \infty.
 \ee
  Let $\{\bu_n\}_{n\in \mathbb{N}}$ be a sequence in $\fTB^I$.
   Due to \eqref{arzela} and the compact embedding of $V$ in $H$,
we can invoke an Arzela-Ascoli type theorem and a diagonal process,
to extract a subsequence $\bu_{n_k} \lra \bu$ in $\Cloc(I; H)$ as $k \ra \infty$.
Moreover, due to the uniform bounds on $\|A\bu_{n_k}(t)\|_{L^2}$ and
$\|\nabla \bu_{n_k}(t)\|_{L^2}$ given in \eqref{defTB}, and since
$\bu_{n_k}  \lra u$ in $C_{loc}(I;H)$ as $k \ra \infty$, one can easily show that
$\|A\bu(t)\|_{L^2}$ and $\|\nabla \bu_{n_k}(t)\|_{L^2}$ satisfy the same bounds given in
\eqref{defTB}. Thus, $\bu \in \fT^I_b$ and
we conclude that $\fTB^I$ is compact in $\Cloc(I,H)$.


 Let now $K \subset I$ be a compact subinterval. By interpolation, we have
 \[
  \|\nabla\bu_{n_k}(t) - \nabla\bu(t) \|_{L^2} \leq c \|\bu_{n_k}(t) - \bu(t)\|_{L^2}^{1/2} \|A\bu_{n_k}(t) - A \bu(t)\|_{L^2}^{1/2} \quad \forall t \in K.
 \]
 Thus,
 \[
  \sup_{t\in K} \|\nabla\bu_{n_k}(t) - \nabla\bu(t) \|_{L^2} \leq c (2  c_2 \nu \kappa_0^2 (G + c_L^{-2})^3)^{1/2} \sup_{t\in K} \|\bu_{n_k}(t) - \bu(t)\|_{L^2}^{1/2}.
 \]
 Since $\bu_{n_k} \to \bu$ in $\Cloc(I;H)$, the right-hand side of the above inequality vanishes as $k \to \infty$, and we obtain that $\bu_{n_k} \to \bu$ in $\Cloc(I;V)$. Therefore, $\fTB^I$ is compact in $\Cloc(I;V)$.
\end{proof}

Here, we consider trajectory statistical solutions on $\R_+$ that are carried by $\fTB^+$. This assumption of being carried by $\fTB^+$ is needed in order to make sense of the measure $(W_+ \circ \mJ) \mu$, with $\mJ$ corresponding to an interpolant operator $J$ which is either a Type I interpolant operator satisfying \eqref{h1type1} or a Type II interpolant operator satisfying \eqref{h1type2}. More specifically, the requirement that trajectories in $\fTB^+$ are uniformly bounded in time with respect to the norm in $V$ is needed  when $J$ is as in the former case; while the requirement that the trajectories are uniformly bounded in time with respect to the norm in $D(A)$ is needed in addition when $J$ is as in the latter case. This is clearly seen from the definitions of the map $W_+$ and the types of interpolant operator $J$.

Notice that, for every solution $\bu \in \fT^{\R_+}$, there exists $t_0 = t_0 (\nu, \kappa_0, G,$ $\|\nabla \bu(0)\|_{L^2})$ such that
\be\label{boundfTBV}
 \|\nabla \bu(t)\|_{L^2} \leq \sqrt{2} \nu \kappa_0 G \quad \forall t \geq t_0,
\ee
and
\be\label{boundfTBDA}
 \|A\bu(t)\|_{L^2} \leq c_2 \nu \kappa_0^2 (G + c_L^{-2})^3  \quad \forall t \geq t_0.
\ee
Therefore, $\tau_{t_0} \bu \in \fTB^+$. The bound in \eqref{boundfTBV} is easily seen by taking the inner product of \eqref{eqNSE} in $H$ with $A\bu$ and performing the usual estimates, while the bound in \eqref{boundfTBDA} follows analogously to the proof of \eqref{boundmADA} in \cite{FoiasJollyLanRupamYangZhang2015}.

Therefore, if we consider an initial measure $\mu_0$ which is carried by $B_V(R)$, for some $R > 0$, then  by \eqref{boundfTBV}, \eqref{boundfTBDA},
the trajectory statistical solution on $\R_+$ starting from $\mu_0$, i.e., $\mu = \mS \mu_0$, satisfies $\tau_{t_0} \mu (\fTB^+) = 1$, for some $t_0 = t_0(\nu, \kappa_0, G, R)$. Indeed,
\[
 \tau_{t_0} \mu (\fTB^+) = \mu (\tau_{t_0}^{-1} \fTB^+) = \mS \mu_0 (\tau_{t_0}^{-1} \fTB^+) = \mu_0 (\mS^{-1} \circ \tau_{t_0}^{-1} (\fTB^+)) \geq \mu_0 (B_V(R)) =1.
\]
Therefore, $\tau_{t_0} \mu$ is a trajectory statistical solution on $[0,\infty)$ which is carried by $\fTB^+$.

Another fact that   we need to verify in order to make sense of the measure $(W_+ \circ \mJ) \mu$ and, in addition, its translations or evaluations in time, is the measurability of the mappings $W_+$, $\mJ$, $\tau_\sigma$ and $\mE_t$, with $\sigma,t \geq 0$. More specifically, this needs to be proved by considering the corresponding domain and range spaces endowed with the topology of uniform convergence on compact sets (see subsection \ref{subsecSpacContFunc}), since this is the natural topology to be considered in the context of statistical solutions.  This is done in Lemma \ref{lemmeasurability} below. For this, we assume throughout this section that $J$ is either a Type I interpolant operator satisfying \eqref{type1} and \eqref{h1type1} or a Type II interpolant satisfying the stronger condition \eqref{type2b}, as well as \eqref{h1type2}. The stronger condition \eqref{type2b} is needed for establishing continuity of $W_+$, and consequently its measurability, between appropriate spaces.


From now on, for every interval $I \subset \R$, we denote by $\fTBH^I$ the space $\fTB^I$ endowed with the topology  inherited from $\Cloc(I;H)$, and by $\fTBV^I$ when $\fT^I$ is endowed with the topology  inherited from $\Cloc(I;V)$. Moreover, we denote by $B_X(\rho)_\locH$ and $B_{X_+}(\rho)_\locH$ the balls $B_X(\rho)$ and $B_{X_+}(\rho)$ endowed with the induced topology from $\Cloc(\R,H)$ and $\Cloc(\R_+,H)$, respectively.

\begin{lem}\label{lemmeasurability}
 The following hold:
 \begin{enumerate}[(i)]
  \item\label{lemmeasurabilityia} For every $t \geq 0$  and $I \subset \R$,  $\mE_t: \Cloc(I,H) \to H$ is a continuous function.
  \item\label{lemmeasurabilityi} For every $\sigma \geq 0$, $\tau_\sigma: \Cloc(\R_+;H) \to \Cloc(\R_+;H)$ is a continuous function.
  \item\label{lemmeasurabilityii}  Let $J$ be an interpolant operator satisfying either \eqref{type1} or \eqref{type2b}. Also, consider $\rho > 0$, $\beta > 0$ and $h > 0$ satisfying conditions \eqref{condbeta} and \eqref{condbetah}. Then,
    \[
     W: B_X(\rho)_\locH \to \Cloc(\R;V)
    \]
    and
    \[
     W_+: B_{X_+}(\rho)_\locH \to \Cloc(\R_+;V)
    \]
    are continuous functions. Moreover, $W_+$ is a Lipschitz function with respect to the metrics $d_0^+$ in $B_{X_+}(\rho)_\locH$ and $d_1^+$ in $\Cloc(\R_+;V)$.
  \item\label{lemmeasurabilityiv} Let $J$ be an interpolant operator satisfying either \eqref{type1} or \eqref{type2b}. Then, the mapping $\mJ: \fTBH^I \to \Cloc(I;(\dot{L}^2(\Omega))^2)$, defined as in \eqref{defmJ}, is continuous.
 \end{enumerate}
\end{lem}
\begin{proof}
%
The proofs of \eqref{lemmeasurabilityia} and \eqref{lemmeasurabilityi} follow directly from the definition of the topology in $\Cloc(\R_+;H)$ and the definitions of the mappings $\mE_t$ and $\tau_\sigma$.

Now let us prove \eqref{lemmeasurabilityii}. Let $\bv_1,\bv_2 \in B_X(\rho)$ and consider $\varepsilon >0$ and $K =[a,b] \subset \R$, a compact interval. Denote $\tv = \bv_2 - \bv_1$ and $\tw = W(\bv_2) - W(\bv_1)$,  and proceed as in the proof of item \eqref{estdiffw1} of Proposition \ref{propexistuniqw}. Then, estimating the third term on the right-hand side of the first inequality in \eqref{ineqwtilde} as
\[
 \beta \nu \kappa_0^2 |(\tv,A\tw)| \leq 2 \beta^2 \nu \kappa_0^4 \|\tv\|_{L^2}^2 + \frac{\nu}{8} \|A\tw\|_{L^2}^2,
\]
and proceeding analogously as in \eqref{ineqwtilde}-\eqref{estdiffw3a}, one obtains that, for every $t, \sigma \in \R$ with $\sigma < t$,
\begin{multline}\label{ineqtwlemmeas}
 \|\nabla\tw(t)\|_{L^2}^2 \leq \|\nabla\tw(\sigma)\|_{L^2}^2  \Exp^{- \frac{\beta \nu \kappa_0^2}{2}(t - \sigma) }+ 8 \beta \kappa_0^2 \sup_{s \in [\sigma,t]}\|\tv(s)\|_{L^2}^2 \\
 \leq 2 (\nu \kappa_0)^2 \left( \frac{G^2}{\beta} + \rho^2 \right)  \Exp^{- \frac{\beta \nu \kappa_0^2}{2}(t-\sigma)} + 8 \beta \kappa_0^2 \sup_{s \in [\sigma,b]}\|\tv(s)\|_{L^2}^2,
\end{multline}
where in the last inequality we used item \eqref{estw1} of Proposition \ref{propexistuniqw}. In particular, let us choose $t \in K = [a,b]$ and $\sigma <a$, with sufficiently large absolute value such that
\be\label{ineqcondsigmalemmeas}
 2 (\nu \kappa_0)^2 \left( \frac{G^2}{\beta} + \rho^2 \right)  \Exp^{ -\frac{\beta \nu \kappa_0^2}{2}(a-\sigma)} < \frac{\varepsilon}{2}.
\ee
Thus, if $\bv_2 \in \bv_1 + \mN_1$, with
\[
 \mN_1 = \{ \bv \in B_X(\rho) \,:\, \|\bv(s)\|_{L^2} < (16 \beta \kappa_0^2)^{-1}\varepsilon, \,\,\, \forall s \in [\sigma,b] \},
\]
then, from \eqref{ineqtwlemmeas} and \eqref{ineqcondsigmalemmeas}, it follows that
\[
 \sup_{t \in [a,b]} \|\nabla [W(\bv_2) - W(\bv_1)](t)\|_{L^2} < \varepsilon.
\]Moreover, we also show the measurability of the determining map $W$ since this is needed in Subsection \ref{subsecDetParamStatSol} when we consider trajectory statistical solutions on the whole $\R$.
This shows that $W: B_X(\rho)_\locH \to \Cloc(\R;V)$ is continuous.

In order to prove that $W_+: (B_{X_+}(\rho),d_0^+) \to (\Cloc(\R_+;V), d_1^+)$ is Lipschitz, consider $\bv_1,\bv_2 \in B_{X_+}(\rho)$. As before, denote $\tv = \bv_1 - \bv_2$ and $\tw = W_+(\bv_2) - W_+(\bv_1)$. Similarly as in \eqref{ineqtwlemmeas}, we have, for every $t, \sigma \in \R_+$,
\[
 \|\nabla\tw(t)\|_{L^2}^2 \leq \|\nabla\tw(\sigma)\|_{L^2}^2 \Exp^{- \frac{\beta \nu \kappa_0^2}{2}(t - \sigma)} + 8 \beta \kappa_0^2 \sup_{s \in [\sigma,t]}\|\tv(s)\|_{L^2}^2.
\]
In particular, choosing $\sigma = 0$ and $t \in [0, n(\nu \kappa_0^2)^{-1}]$, with $n\in \mathbb{N}$, we obtain that
\[
 \sup_{t \in [0, n(\nu \kappa_0^2)^{-1}]} \|\nabla\tw(t)\|_{L^2}^2 \leq 8 \beta \kappa_0^2 \left(\sup_{s \in [0, n(\nu \kappa_0^2)^{-1}]}\|\tv(s)\|_{L^2}^2\right),
\]
where we used that $\tw(0) = 0$. Thus,
\be\label{estWLip}
 \sup_{s\in K_n} \frac{\|\nabla[W_+(\bv_2) - W_+(\bv_1)](s)\|_{L^2}}{\nu \kappa_0} \leq (8 \beta)^{1/2} \sup_{s\in K_n} \frac{\|\bv_2(s) - \bv_1(s)\|_{L^2}}{\nu} \quad \forall n \in \mathbb{N},
\ee
where $K_n = [0, n(\nu \kappa_0^2)^{-1}]$. From the definitions of the metrics $d_0^+$ and $d_1^+$, this implies that
\be\label{ineqW+Lipschitz}
 d_1^+(W_+(\bv_2),W_+(\bv_1)) \leq (8 \beta)^{1/2} d_0^+(\bv_2,\bv_1),
\ee
as desired.

Finally, let us prove \eqref{lemmeasurabilityiv}. Suppose that $J$ is a Type II interpolant operator satisfying the stronger property \eqref{type2b}. Let $\bu_1, \bu_2 \in \fTB^I$. Notice that, for every $t,s \in I$,
\begin{multline}
 \|(\mJ \bu_1)(t) - (\mJ \bu_2)(s)\|_{L^2} \leq \|\bu_1(t) - \bu_2(s)\|_{L^2} + c_{2,1}' h \|\nabla\bu_1(t) - \nabla\bu_2(s)\|_{L^2} \\
 + c_{2,2}' h^{3/2} \|\nabla\bu_1(t) - \nabla\bu_2(s)\|_{L^2}^{1/2} \|A\bu_1(t) - A \bu_2(s)\|_{L^2}^{1/2} \\
 \leq \|\bu_1(t) - \bu_2(s)\|_{L^2} + c_{2,1}' h  \|\bu_1(t) - \bu_2(s)\|_{L^2}^{1/2}
 \|A\bu_1(t) - A\bu_2(s)\|_{L^2}^{1/2} \\
+ c_{2,2}' h^{3/2} \|\bu_1(t) - \bu_2(s)\|_{L^2}^{1/4} \|A\bu_1(t) - A \bu_2(s)\|_{L^2}^{3/4}\\
\leq \|\bu_1(t) - \bu_2(s)\|_{L^2} + c_{2,1}' h
 (2 c_2 \nu \kappa_0^2 (G + c_L^{-2})^3)^{1/2}\|\bu_1(t) - \bu_2(s)\|_{L^2}^{1/2} \\
+ c_{2,2}' h^{3/2} (2 c_2 \nu \kappa_0^2 (G + c_L^{-2})^3)^{3/4} \|\bu_1(t) - \bu_2(t)\|_{L^2}^{1/4}. \label{intermedineq}
\end{multline}
First, by taking $\bu_1=\bu_2=\bu$ in \eqref{intermedineq}, we see that $\mJ \bu$ belongs to $\Cloc(I,$ $ (L^2(\Omega))^2)$. Subsequently, by taking $s=t$ in \eqref{intermedineq}, it follows that
$\mJ: \fTBH^I \to \Cloc(I, (\dot{L}^2(\Omega))^2)$ is continuous.
The proof for an interpolant operator satisfying \eqref{type1} is similar, and therefore omitted.
\end{proof}

%

In the next theorem, for a given trajectory statistical solution $\mu$ on $\R_+$, we show the asymptotic convergence in time of the translations or evaluations in time of $(W_+ \circ \mJ) \mu$ to $\mu$, with respect to the Kantorovich metric. Note that the measure
$\mJ \mu$ is constructed from the observations. For instance, if $J$ is the nodal interpolant, then the corresponding measure $\mJ \mu$ is constructed from the statistics observed at the measurement nodes.
Applying $W_+$ to the measure $\mJ \mu$, \comments{which is constructed from the uncertainties inherent to the measurements,} constitutes  an ensemble-based data assimilation algorithm, where the evolution equation \eqref{dataeqn} is used as the forecast model.

We recall that $\Gamma_H$ denotes the Kantorovich metric on $\mP_1(\Cloc(\R_+;H),d_0^+)$ and $\gamma_H$ denotes the Kantorovich metric on $\mP_1(H,d_H)$, with $d_H(\bu,\bv) = \|\bu - \bv\|_{L^2}$. Moreover, we recall that $\mu_{\fTB^+}$ denotes the restriction of the measure $\mu$ to the measurable set $\fTB^+$, as defined in \eqref{restrictedsigmaalg}. In case $\mu$ is carried by $\fTB^+$ then, for any Borel set $E$, $\mu_{\fTB^+} (E \cap \fTB^+) = \mu(E)$.

\begin{thm}\label{thmDAmeasTB}
 Let $J$ be an interpolant operator satisfying either \eqref{type1} and \eqref{h1type1} or \eqref{type2b} and \eqref{h1type2}. Let $\rho > 0$, $\beta > 0$ and $h >0$ satisfying conditions \eqref{condbeta} and \eqref{condbetah}. Moreover, in case $J$ satisfies \eqref{type1} and \eqref{h1type1}, we assume that $\rho$ satisfies
 \be\label{condrhoballa}
  \rho \geq \sqrt{2}(1+\tilde{c}_1) G,
 \ee
where $\tilde{c}_1$ is the constant from \eqref{h1type1}; and in case $J$ satisfies \eqref{type2b} and \eqref{h1type2}, we assume that
\be\label{condrhoballb}
 \rho \geq c_3^{**} \left[ G + \frac{(G + c_L^{-2})^3}{\beta^{1/2}}\right],
\ee
where $c_3^{**} = \max\{ \sqrt{2} (1 + \tilde{c}_{2,1}), \tilde{c}_{2,2} c_2 \sqrt{c_2^*}\}$, with $\tilde{c}_{2,1}$ and $\tilde{c}_{2,2}$ being the constants from \eqref{h1type2}, $c_2$ the constant from \eqref{boundmADA} and $c_2^*$ the constant from \eqref{condbetah}. Let $\mu \in \mP(\Cloc(\R_+;H),d_0^+)$ be a trajectory statistical solution on $\R_+$ which is carried by $\fTB^+$. Then, the following properties hold:
 \begin{enumerate}[(i)]
  \item\label{thmDAmeasTBi} $\Gamma_H ((\tau_t \circ W_+ \circ \mJ)\mu|_{\fTB^+}, \tau_t \mu) \to 0$ exponentially, as $t \to \infty$;
  \item\label{thmDAmeasTBii} $\gamma_H ((\mE_t \circ W_+ \circ \mJ)\mu_{\fTB^+}, \mE_t \mu) \to 0$ exponentially, as $t \to \infty$.
 \end{enumerate}
\end{thm}
\begin{proof}
 Let $\Phi: (\Cloc(\R_+;H),d_0^+) \to \R$ be a Lipschitz function with $\|\Phi\|_{\Lip} \leq 1$. Notice that
 \begin{multline}\label{estdiffmeasmuWJmu}
  \left| \int_{\Cloc(\R_+;H)} \Phi(\bu) \rd (\tau_t \circ W_+ \circ \mJ) \mu_{\fTB^+}(\bu) - \int_{\Cloc(\R_+;H)} \Phi(\bu) \rd (\tau_t \mu)(\bu)\right| \\
  = \left| \int_{\fTB^+} \Phi((\tau_t\circ  W_+ \circ \mJ) \bu) \rd \mu(\bu) - \int_{\fTB^+} \Phi(\tau_t \bu) \rd \mu (\bu)\right| \\
  \leq \int_{\fTB^+} d_0^+((\tau_t \circ W_+ \circ \mJ) \bu, \tau_t \bu) \rd \mu (\bu) \leq  \int_{\fTB^+} \sup_{s \geq 0} \|[\tau_t \circ W_+ \circ \mJ(\bu) - \tau_t \bu](s)\|_{L^2} \rd \mu(\bu) \\
  \leq  \int_{\fTB^+} \sup_{s\geq t} \|[W_+\circ  \mJ (\bu) - \bu](s)\|_{L^2} \rd \mu(\bu) \leq \sqrt{2} \nu \kappa_0 G \Exp^{-\frac{\beta \nu \kappa_0^2}{4}t},
 \end{multline}
where, in the last inequality we used item \eqref{thmpropsW+iiip} of Theorem \ref{thmpropsW+}. Since $\Phi$ is an arbitrary Lipschitz function on $(\Cloc(\R_+;H),d_0^+)$ with $\|\Phi\|_{\Lip} \leq 1$, we conclude that
\[
 \Gamma_H ((\tau_t\circ  W_+ \circ \mJ)\mu_{\fTB^+}, \tau_t \mu) \leq \sqrt{2} \nu \kappa_0 G \Exp^{-\frac{\beta \nu \kappa_0^2}{4}t},
\]
which proves \eqref{thmDAmeasTBi}.

In order to prove \eqref{thmDAmeasTBii}, let $\varphi: H \to \R$ be a Lipschitz function with $\|\varphi\|_{\Lip} \leq 1$. Notice that
\begin{multline}
 \left| \int_H \varphi(\bu) \rd (\mE_t\circ  W_+\circ \mJ) \mu_{\fTB^+}(\bu) - \int_H \varphi(\bu) \rd (\mE_t \mu)(\bu)\right| \\
 = \left| \int_{\fTB^+} \varphi(\mE_t \circ W_+ \circ \mJ (\bu) \rd \mu (\bu) - \int_{\fTB^+} \varphi(\mE_t \bu) \rd \mu(\bu)\right| \\
 \leq \int_{\fTB^+} \|\mE_t \circ W_+ \circ \mJ( \bu) - \mE_t\bu\|_{L^2} \rd \mu(\bu) \leq \int_{\fTB^+} \sup_{s\geq t}\|[W_+\circ \mJ (\bu) - \bu](s)\|_{L^2} \rd \mu(\bu).
\end{multline}

The remaining of the proof follows analogously as in the previous item.
\end{proof}

The next result shows that, when $J$ is either a Type I interpolant operator or a Type II interpolant operator satisfying the stronger property \eqref{type2b}, then we can prove that, given a trajectory statistical solution $\mu$ on $\R_+$, the asymptotic convergence of translations or evaluations in time of $(W_+ \circ \mJ)\mu$ to $\mu$ is valid in the sense of distributions. Observe that convergence in the Kantorovich metric implies convergence in distribution (i.e., weak convergence of probability measures) provided the underlying measures are carried by a compact set (see, e.g., \cite[Theorem 11.3.3]{Dudley2002}). Under the assumption that $J$ is a Type I interpolant, or a Type II interpolant satisfying the stronger condition \eqref{type2b}, the associated mapping $\mJ$ is continuous between appropriate spaces. Therefore, the push forward of measures carried by a compact set is also carried by a compact set. Thus, convergence in the Kantorovich metric implies convergence in distribution in this case. This is stated in the Corollary below.
\begin{cor}\label{corDAmeasTB}
 Let $J$ be an interpolant operator satisfying either \eqref{type1} and \eqref{h1type1} or \eqref{type2b} and \eqref{h1type2}. Let $\rho > 0$, $\beta > 0$ and $h >0$ satisfying conditions \eqref{condbeta} and \eqref{condbetah}. Moreover, in case $J$ satisfies \eqref{type1} and \eqref{h1type1}, assume that $\rho$ satisfies \eqref{condrhoballa}; and in case $J$ satisfies \eqref{type2b} and \eqref{h1type2}, assume that $\rho$ satisfies \eqref{condrhoballb}. Let $\mu \in \mP(\Cloc(\R_+;H),d_0^+)$ be a trajectory statistical solution on $\R_+$ which is carried by $\fTB^+$. Then, the following properties hold:
 \begin{enumerate}[(i)]
  \item\label{corDAmeasTBi} For every continuous function $\Phi: \Cloc(\R_+,H) \to \R$,
   \begin{multline}
    \lim_{t \to \infty} \left| \int_{\Cloc(\R_+;H)} \Phi(\bu) \rd (\tau_t \circ W_+ \circ \mJ) \mu_{\fTB^+}(\bu) - \int_{\Cloc(\R_+;H)} \Phi(\bu) \rd (\tau_t \mu)(\bu)\right|\\
     = 0;
   \end{multline}
  \item\label{corDAmeasTBii} For every continuous function $\varphi: H \to \R$,
   \be
    \lim_{t \to \infty} \left| \int_H \varphi(\bu) \rd (\mE_t\circ  W_+\circ \mJ) \mu_{\fTB^+}(\bu) - \int_H \varphi(\bu) \rd (\mE_t \mu)(\bu)\right| = 0.
   \ee
 \end{enumerate}
\end{cor}

\begin{rmk}
Note that for any Borel measure $\eta$ on $\Cloc(\R_+,H)$, $\tau_t \eta$ is the time shifted measure on paths starting at time $t$. The above theorem says that the measures
$W_+ \circ \mJ \circ \mu$, which is constructed by the data assimilation algorithm from the observed measure $\mJ \circ \mu$, and $\mu$ converge when shifted by time $t$ as $t \ra \infty$. This is precisely the data assimilation algorithm for observed path space measures.
\end{rmk}

\subsection{Determining parameters for statistical solutions}\label{subsecDetParamStatSol}

In this subsection, we show that statistical solutions of the 2D NSE can be determined by a finite number of parameters, which are represented through the finite-rank interpolant operator $J$ (of Type I or Type II). This type of result, as is well-known, holds for individual solutions of the 2D NSE, and thus it is natural to expect it to hold for statistical solutions as well. The concept of data assimilation by the ``nudging approach"
initiated in \cite{AzouaniOlsonTiti2014} was motivated by the existence of determining parameters for dissipative systems. However,
as in the case of individual trajectories for the MHD equation in \cite{BHLP},
here we use the data assimilation algorithm for statistical solutions to establish the existence
of determining parameters for statistical solutions.


We start by showing this type of result for trajectory statistical solutions on $\R$ which are carried by the set of trajectories in the global attractor $\mA$, denoted by

\be\label{defTA}
\fTA = \{\bu \in \fT^{\R} \,:\, \bu(t) \in \mA \,\,\, \forall t \in \R \}.
\ee

\comments{Notice that $\fTA$ is contained in the compact set $\mS(B_H(\nu G))$. Hence, since $\fTA$ is also closed in $\Cloc(I,H)$, it follows in particular that $\fTA$ is compact and thus a Borel set in $\Cloc(I,H)$. Moreover,}  Similarly as in the proof of Lemma \ref{lemTBIcptCIV}, we can prove that $\fTA$ is compact in $\Cloc(\R,V)$.
Thus, $\fTA$ is compact, and therefore a Borel subset, of $\Cloc(\R,H)$.

\begin{thm}\label{thmDAmeasA}
 Let $J$ be an interpolant operator satisfying either \eqref{type1} and \eqref{h1type1} or \eqref{type2b} and \eqref{h1type2}. Let $\rho > 0$, $\beta > 0$ and $h >0$ satisfying conditions \eqref{condbeta} and \eqref{condbetah}. In case $J$ satisfies \eqref{type1} and \eqref{h1type1}, assume that $\rho$ satisfies \eqref{condrhoa}; and in case $J$ satisfies \eqref{type2b} and \eqref{h1type2}, assume that $\rho$ satisfies \eqref{condrho}. Let $\mu$ and $\eta$ be two trajectory statistical solutions carried by $\fTA$. If $\mJ \mu_{\fTA} = \mJ \eta_{\fTA}$, then $\mu = \eta$.
\end{thm}
\begin{proof}
Let $\Phi \in \Cb(\Cloc(\R;H))$. It suffices to prove that (see, e.g., \cite[Lemma 9.3.2]{Dudley2002})
\be\label{eqmeasmueta}
 \int_{\Cloc(\R;H)} \Phi(\bu) \rd \mu(\bu) = \int_{\Cloc(\R;H)} \Phi(\bu) \rd \eta(\bu).
\ee
Observe that due to \eqref{condrhoa} or \eqref{condrho}, as well as \eqref{boundmAH1} and \eqref{boundmADA}, $\mJ u \in B_\rho(X)$ for $u \in \fTA$.
Thus, $W \circ \mJ $ is defined for $\bu \in \fTA$.
Since $\mu$ is carried by $\fTA$, using item \eqref{thmpropsWiv} of Theorem \ref{thmpropsW}, we obtain that
\begin{multline}\label{eqintmu}
 \int_{\Cloc(\R;H)} \Phi(\bu) \rd \mu(\bu) =  \int_{\fTA} \Phi(\bu) \rd \mu(\bu) =  \int_{\fTA} \Phi(W \circ \mJ (\bu)) \rd \mu(\bu)\\
 =  \int_{\mJ(\fTA)} \Phi(W (\bv)) \rd (\mJ\mu)(\bv).
\end{multline}

Since, by hypothesis, $\mJ\mu_{\fTA} = \mJ \eta_{\fTA}$, then
\be
 \int_{\mJ(\fTA)} \Phi(W (\bv)) \rd (\mJ\mu)(\bv) = \int_{\mJ(\fTA)} \Phi(W (\bv)) \rd (\mJ\eta)(\bv)
\ee

But, analogously to \eqref{eqintmu}, we have
\be\label{eqinteta}
 \int_{\mJ(\fTA)} \Phi(W (\bv)) \rd (\mJ\eta)(\bv) =  \int_{\Cloc(\R;H)} \Phi(\bu) \rd \eta(\bu)
\ee

Thus, from \eqref{eqintmu}-\eqref{eqinteta}, we conclude \eqref{eqmeasmueta}.
\end{proof}

\begin{rmk}
We observe that if $\mu$ is a trajectory statistical solution on $\R$ which is invariant under the action of the semigroup of translations $\{\tau_t\}_{t \geq 0}$ (i.e., $\tau_t \mu = \mu$, for all $t \geq 0$) then $\mu$ is carried by $\fTA$. Indeed, for every $t \geq 0$, we have $(S(t) \circ \mE_0) \mu = \mE_t \mu =  (\mE_0 \circ \tau_t )\mu = \mE_0 \mu$. Thus, $\mE_0 \mu$ is an invariant measure with respect to the semigroup $\{ S(t)\}_{t \geq 0}$, which implies that $\mE_0 \mu (\mA) = 1$ (\cite[Theorem IV.4.1]{FMRT2001}). Notice that, by the definition of $\mA$ (in particular, item 4), it follows that the mapping  $\mE_0|_{\fTA}:\fTA \ra \mA$ is surjective. Consequently, we have
\[
 \mu (\fTA) = \mu(\mE_0^{-1} \circ \mE_0 (\fTA)) \geq \mu(\mE_0^{-1} (\mA)) = \mE_0 \mu (\mA) = 1.
\]
\end{rmk}

Now, in the following theorem, we consider the case of trajectory statistical solutions on $\R_+$ which are carried by the set $\fTB^+$, given in \eqref{defTB}. We show that the translations in time of such measures can be determined by a finite number of parameters.

Recall that $\Gamma_H$ denotes the Kantorovich metric on $\mP_1(\Cloc(\R_+,H),d_0^+)$.

\begin{thm}\label{thmdetquantmeas}
 Let $J$ be an interpolant operator satisfying either \eqref{type1} and \eqref{h1type1} or \eqref{type2b} and \eqref{h1type2}. Let $\rho > 0$, $\beta > 0$ and $h >0$ satisfying conditions \eqref{condbeta} and \eqref{condbetah}. Moreover, in case $J$ satisfies \eqref{type1} and \eqref{h1type1}, assume that $\rho$ satisfies \eqref{condrhoballa}; and in case $J$ satisfies \eqref{type2b} and \eqref{h1type2}, assume that $\rho$ satisfies \eqref{condrhoballb}. Let $\mu, \eta \in \mP(\Cloc(\R_+;H),d_0^+)$ be two trajectory statistical solutions on $\R_+$ which are carried by $\fTB^+$. If
 \be\label{thmdetquantmeashyp}
  \lim_{t \to \infty}  \Gamma_H( (\tau_t \circ \mJ) \mu_{\fTB^+}, (\tau_t \circ \mJ) \eta_{\fTB^+}) = 0,
 \ee
 then
 \be\label{thmdetquantmeasres}
  \lim_{t \to \infty}  \Gamma_H ( \tau_t \mu, \tau_t \eta) = 0.
 \ee
\end{thm}
\begin{proof}
Notice that
\begin{multline}\label{rewritekappa0}
  \Gamma_H( \tau_t \mu, \tau_t \eta) \leq  \Gamma_H( \tau_t \mu, (\tau_t\circ  W_+ \circ \mJ) \mu_{\fTB^+}) +  \Gamma_H( (\tau_t\circ W_+ \circ \mJ)\mu_{\fTB^+}, (\tau_t\circ W_+ \circ \mJ) \eta_{\fTB^+}) \\
  +  \Gamma_H( \tau_t \eta, (\tau_t\circ  W_+\circ \mJ) \eta_{\fTB^+}).
\end{multline}

Let $\Phi: (\Cloc(\R_+;H),d_0^+) \to \R$ be a Lipschitz function with $\|\Phi\|_{\Lip} \leq 1$.

Using inequality \eqref{ineqW+Lipschitz}  for $W_+$ and the Poincar\'e inequality, we have
\be\label{ineqW+Lipschitzb}
 d_0^+(W_+(\bv_1), W_+(\bv_2)) \leq (8 \beta)^{1/2} d_0^+(\bv_1,\bv_2) \quad \forall \bv_1, \bv_2 \in B_{X_+}(\rho).
\ee
This implies that $\Phi \circ W_+: (B_{X_+}(\rho),d_0^+) \to \R$ is a Lipschitz function with $\|\Phi \circ W_+\|_{\Lip} \leq (8 \beta)^{1/2}$. Therefore,
\begin{multline}
 \left| \int_{\Cloc(\R_+;H)} \Phi(\bu) \rd (W_+ \circ \tau_t \circ \mJ) \mu_{\fTB^+}(\bu) - \right.\\
 \left. \int_{\Cloc(\R_+;H)} \Phi(\bu) \rd (W_+\circ \tau_t\circ \mJ) \eta_{\fTB^+}(\bu)  \right| \\
 =  \left| \int_{B_{X_+}(\rho)} \Phi(W_+ (\bv)) \rd (\tau_t\circ \mJ) \mu_{\fTB^+}(\bv) - \int_{B_{X_+}(\rho)} \Phi(W_+ (\bv)) \rd ( \tau_t\circ \mJ) \eta_{\fTB^+}(\bv)  \right|\\ \leq (8 \beta)^{1/2} \Gamma_H ( (\tau_t\circ \mJ) \mu_{\fTB^+}, (\tau_t \circ \mJ) \eta_{\fTB^+}).
\end{multline}
Since $\Phi$ is an arbitrary Lipschitz function with $\|\Phi\|_{\Lip} \leq 1$, we then have
\begin{multline}\label{ineqkappa0c}
 \Gamma_H( (\tau_t\circ  W_+\circ \mJ) \mu_{\fTB^+}, (\tau_t\circ W_+\circ \mJ) \eta_{\fTB^+}) = \Gamma_H(( W_+ \circ\tau_t \circ \mJ) \mu_{\fTB^+}, (W_+ \circ \tau_t \circ \mJ) \eta_{\fTB^+}) \\
 \leq (8 \beta)^{1/2} \Gamma_H ( (\tau_t\circ \mJ) \mu_{\fTB^+}, (\tau_t \circ \mJ) \eta_{\fTB^+}),
\end{multline}
where we used that $W_+$ commutes with $\tau_t$ (see Theorem \ref{thmpropsW+}, \eqref{thmpropsW+v}).

Thus, \eqref{thmdetquantmeasres} follows from \eqref{rewritekappa0}, by using item \eqref{thmDAmeasTBi} of Theorem \ref{thmDAmeasTB}, \eqref{ineqkappa0c} and hypothesis \eqref{thmdetquantmeashyp}.
\end{proof}

Finally, still considering the case of trajectory statistical solutions on $\R_+$ which are carried by $\fTB^+$, we now prove that its evaluations in time can be determined by a finite number of parameters.


\begin{thm}\label{cordetquantmeas}
 Under the hypotheses of Theorem \ref{thmdetquantmeas}, if
 \be\label{cordetquantmeashyp}
   \lim_{t \to \infty}  \Gamma_H( (\tau_t\circ \mJ) \mu_{\fTB^+}, (\tau_t\circ \mJ) \eta_{\fTB^+}) = 0,
 \ee
 then
 \be\label{cordetquantmeasres}
   \lim_{t \to \infty}  \gamma_H ( \mE_t \mu, \mE_t \eta) = 0.
 \ee
\end{thm}
\begin{proof}
We have
\begin{multline}\label{rewritekappaH}
 \gamma_H ( \mE_t \mu, \mE_t \eta) \leq \gamma_H ( \mE_t \mu, (\mE_t\circ W_+\circ \mJ) \mu_{\fTB^+}) \\
 + \gamma_H ((\mE_t\circ W_+\circ \mJ) \mu_{\fTB^+}, (\mE_t\circ W_+\circ \mJ) \eta_{\fTB^+}) 
 + \gamma_H ( \mE_t \eta, (\mE_t\circ W_+\circ \mJ) \eta_{\fTB^+} ).
\end{multline}

Notice that
\be\label{rewriteEt}
 \mE_t = \mE_t \circ \tau_{-t} \circ \tau_t = \mE_0 \circ \tau_t.
\ee

Let $\varphi: H \to \R$ be a Lipschitz function with $\|\varphi\|_{\Lip} \leq 1$. Notice that $\varphi \circ \mE_0 \circ W_+: (B_{X_+}(\rho), d_0^+) \to \R$ is a Lipschitz function with
\be\label{ineqphiE0W+Lipsct}
 \|\varphi \circ \mE_0 \circ W_+\|_{\Lip} \leq (8 \beta)^{1/2} \nu (1 + 2 \rho).
\ee
Indeed, by the definition of $d_0^+$, it follows that
\be\label{ineqpropd0+}
 \|\mE_0(\bu_1) - \mE_0(\bu_2)\|_{L^2} \leq \nu(1 + 2 \rho) d_0^+(\bu_1,\bu_2) \quad \forall \bu_1, \bu_2 \in B_{X_+}(\rho).
\ee
Thus, \eqref {ineqphiE0W+Lipsct} follows from \eqref{ineqpropd0+} and \eqref{ineqW+Lipschitzb}.

Now, using that $\tau_t \circ W_+ = W_+ \circ \tau_t$ (cf. Theorem \ref{thmpropsW+}, \eqref{thmpropsW+v}), \eqref{rewriteEt} and \eqref{ineqphiE0W+Lipsct}, we obtain
\begin{multline}
 \left| \int_H \varphi(\bu) \rd (\mE_t \circ W_+ \circ \mJ) \mu_{\fTB^+}(\bu) - \int_H \varphi(\bu) \rd (\mE_t\circ W_+\circ \mJ) \eta_{\fTB^+}(\bu) \right|\\
 = \left| \int_{B_{X_+}(\rho)} \varphi \circ \mE_0 \circ W_+(\bv) \rd (\tau_t\circ \mJ) \mu_{\fTB^+}(\bv) \right.\\
 \left.- \int_{B_{X_+}(\rho)} \varphi \circ \mE_0 \circ W_+(\bv) \rd (\tau_t\circ \mJ) \eta_{\fTB^+}(\bv) \right|\\
 \leq (8\beta)^{1/2} \nu (1 + 2 \rho) \Gamma_H( (\tau_t\circ \mJ) \mu_{\fTB^+}, (\tau_t\circ \mJ) \eta_{\fTB^+}).
\end{multline}
But since $\varphi$ is an arbitrary Lipschitz function with $\|\varphi\|_{\Lip} \leq 1$, we conclude that
\begin{multline}\label{ineqkappaH}
 \gamma_H ((\mE_t\circ W_+\circ \mJ) \mu_{\fTB^+}, (\mE_t\circ W_+\circ \mJ) \eta_{\fTB^+}) \\
 \leq (8\beta)^{1/2} \nu (1 + 2 \rho)\Gamma_H( (\tau_t\circ \mJ) \mu_{\fTB^+}, (\tau_t\circ \mJ) \eta_{\fTB^+}).
\end{multline}
Therefore, \eqref{cordetquantmeasres} follows from \eqref{rewritekappaH}, by using item \eqref{thmDAmeasTBii} of Theorem \ref{thmDAmeasTB}, \eqref{ineqkappaH} and hypothesis \eqref{cordetquantmeashyp}.
\end{proof}

\section*{Acknowledgements}
The authors would like to thank Prof. Ricardo Rosa for the stimulating and helpful discussions. The work of A.B. was supported in part by the NSF grant DMS 1517027, that of C.F. was supported in part by the NSF grant DMS 1516866 and the ONR grant  N00014-15-1-2333. The work of C.F.M. and E.S.T. was partially supported by the ONR grant N00014-15-1-2333.

\setcounter{equation}{0}
\appendix

\section*{Appendix}
\renewcommand{\thesection}{A}
\renewcommand{\theequation}{{A.}\arabic{equation}}

We now show that the example of interpolant operator concerning nodal values, given in \eqref{Jnodes}, is a Type II interpolant operator satisfying the stronger property \eqref{type2b}. The proof follows from a slightly modified version of Proposition 4 in \cite{AzouaniOlsonTiti2014}, whose proof is given in \cite[Appendix]{JonesTiti1993}. In fact, the result below follows the same steps done in \cite[Appendix]{JonesTiti1993}, modulo an application of Young's inequality. We present it here for completeness.

\begin{prop}
 Let $l > 0$ and $Q$ be the square $[0,l] \times [0,l] \subset \R^2$. Then, for every $\varphi \in H^2(Q)$ and $(x_1,x_2), (y_1,y_2) \in Q$, we have
 \be\label{propAppendix}
  |\varphi(x_1,x_2) - \varphi(y_1,y_2)| \leq 2 \left( \|\nabla \varphi\|_{L^2(Q)}^2 + \sqrt{2} l\|\nabla \varphi\|_{L^2(Q)} \left\| \frac{\partial^2 \varphi}{\partial x \partial y} \right\|_{L^2(Q)} \right)^{1/2}
 \ee
\end{prop}
\begin{proof}
 First, consider $\psi = \psi(x,y) \in C^\infty(Q)$ and let $\tilde{y} \in [0,l]$. Without loss of generality, assume that $\tilde{y}$ is closer to $0$ than $l$, i.e., $l - \tilde{y} \geq \tilde{y}$. Notice that, for every $x \in [0,l]$ and $y \in [\tilde{y},l]$, we have
 \[
  \psi^2(x,\tilde{y}) = \psi^2(x,y) - \int_{\tilde{y}}^y \frac{\partial}{\partial y}\psi^2(x,s) \rd s
 \]
 Integrating with respect to $x$ and $y$ over $[0,l]\times[\tilde{y},l]$ and applying Cauchy-Schwarz inequality, it follows that
 \[
  (l - \tilde{y}) \int_0^l |\psi(x,\tilde{y})|^2 \rd x \leq \|\psi\|_{L^2(Q)}^2 + 2(l - \tilde{y}) \|\psi\|_{L^2(Q)} \left\| \frac{\partial \psi}{\partial y}\right\|_{L^2(Q)}.
 \]
 But since $l - \tilde{y} \geq \tilde{y}$, then, in particular, $l - \tilde{y} \geq l/2 > 0$. Therefore,
 \be\label{prelimineq}
   \int_0^l |\psi(x,\tilde{y})|^2 \rd x \leq \frac{2}{l}\|\psi\|_{L^2(Q)}^2 + 2 \|\psi\|_{L^2(Q)} \left\| \frac{\partial \psi}{\partial y}\right\|_{L^2(Q)}.
 \ee
 In case $l - \tilde{y} < l$, we consider $y \in [0,\tilde{y}]$ instead and proceed analogously as above, so that \eqref{prelimineq} is valid for every $\tilde{y} \in [0,l]$. Moreover, since $\psi \in C^\infty(Q)$ is arbitrary, we conclude, by density, that \eqref{prelimineq} is also valid for every $\psi \in H^1(Q)$.

 Now, let $\varphi \in C^\infty(Q)$ and $(x_1,y_1), (x_2,y_2) \in Q$. By triangle inequality,
 \be\label{triangineq}
   |\varphi(x_1,x_2) - \varphi(y_1,y_2)| \leq |\varphi(x_1,y_1) - \varphi(x_2,y_1)| + |\varphi(x_2,y_1) - \varphi(x_2,y_2)|.
 \ee
 Notice that
 \[
  |\varphi(x_1,y_1) - \varphi(x_2,y_1)| = \left|\int_{x_1}^{x_2} \frac{\partial\varphi}{\partial x}(s,y_1) \rd s \right| \leq l^{1/2} \left\| \frac{\partial \varphi}{\partial x}(\cdot,y_1)\right\|_{L^2(Q)}.
 \]
 Thus, using \eqref{prelimineq} with $\psi = \partial \varphi/\partial x$, it follows that
 \be\label{triangineq1}
  |\varphi(x_1,y_1) - \varphi(x_2,y_1)| \leq \left( 2 \left\| \frac{\partial \varphi}{\partial x}\right\|^2_{L^2(Q)} + 2 l \left\| \frac{\partial \varphi}{\partial x}\right\|_{L^2(Q)} \left\| \frac{\partial \varphi}{\partial y \partial x}\right\|_{L^2(Q)} \right)^{1/2}.
 \ee
 Analogously, one can prove that
 \be\label{triangineq2}
  |\varphi(x_2,y_1) - \varphi(x_2,y_2)| \leq \left( 2 \left\| \frac{\partial \varphi}{\partial y}\right\|^2_{L^2(Q)} + 2 l \left\| \frac{\partial \varphi}{\partial y}\right\|_{L^2(Q)} \left\| \frac{\partial \varphi}{\partial x \partial y}\right\|_{L^2(Q)} \right)^{1/2}.
 \ee
 Plugging estimates \eqref{triangineq1} and \eqref{triangineq2} into \eqref{triangineq} and using the density of $C^\infty(Q)$ in $H^2(Q)$, we conclude \eqref{propAppendix}.
\end{proof}

In \cite[Proposition 8]{AzouaniOlsonTiti2014}, it is proved that the interpolant operator given in \eqref{Jnodes} satisfies \eqref{type2}. Now, the proof that this interpolant operator satisfies the stronger property \eqref{type2b} follows the same steps, but using \eqref{propAppendix} instead. We refer the reader to \cite{AzouaniOlsonTiti2014} for further details.

\end{document}